\documentclass[12pt,reqno]{amsart}

\usepackage{fullpage} 
\usepackage{amssymb,amsfonts}
\usepackage[all,arc]{xy}
\usepackage{enumerate}
\usepackage{mathrsfs}
\usepackage{color}   
\usepackage{hyperref}
\hypersetup{
    colorlinks=true, 
    linktoc=all,     
    linkcolor=blue,  
}
\usepackage{amsthm}
\usepackage{amsmath}
\usepackage{graphicx}
\usepackage{pgf,tikz}
\usetikzlibrary{positioning}

\newcommand{\bc}{\mathbb C}

\newcommand{\bz}{\mathbb Z}

\newcommand{\ba}{\mathbb A}

\newcommand{\la}{\langle}
\newcommand{\ra}{\rangle}

\newcommand{\bs}{\backslash}

\newcommand{\al}{\alpha}

\newcommand{\lam}{\lambda}

\DeclareMathOperator{\ind}{Ind}
\DeclareMathOperator{\GL}{GL}
\DeclareMathOperator{\Hom}{Hom}
\DeclareMathOperator{\Ind}{Ind}

\DeclareMathOperator{\SL}{SL}

\newcommand{\fo}{\mathfrak o}

\newcommand{\tgl}{\widetilde{\GL}}
\newcommand{\tgln}{\widetilde{\GL}^{(n)}}
\newcommand{\tm}{\widetilde{M}}
\newcommand{\tmn}{\widetilde{M}^{(n)}}
\newcommand{\tto}{\widetilde{T}}
\newcommand{\tton}{\widetilde{T}^{\square}}
\newcommand{\tb}{\widetilde{B}}

\newcommand{\sco}{\mathcal{O}}

\newtheorem{Thm}{Theorem}[section]
\newtheorem{Prop}[Thm]{Proposition}
\newtheorem{Lem}[Thm]{Lemma}
\newtheorem{Cor}[Thm]{Corollary}

\theoremstyle{definition}
\newtheorem{Def}[Thm]{Definition}

\theoremstyle{remark}
\newtheorem{Rem}[Thm]{Remark}
\newtheorem{Ex}[Thm]{Example}

\theoremstyle{definition}

\title[Fourier Coefficients for Theta Representations]{Fourier Coefficients for Theta Representations on Covers of General Linear Groups}
\author{Yuanqing Cai}
\address{Department of Mathematics, Boston College, Chestnut Hill, MA 02467-3806, USA}
\curraddr{ Department of Mathematics, Weizmann Institute of Science, Rehovot, 7610001, Israel}
\email{yuanqing.cai@weizmann.ac.il}
\date\today

\begin{document}

\begin{abstract}

We show that the theta representations on certain covers of general linear groups support certain types of unique functionals. The proof involves two types of Fourier coefficients. The first are semi-Whittaker coefficients, which generalize coefficients introduced by Bump and Ginzburg for the double cover. The covers for which these coefficients vanish identically (resp. do not vanish for some choice of data) are determined in full. The second are the Fourier coefficients associated with general unipotent orbits.  In particular, we determine the unipotent orbit attached, in the sense of Ginzburg, to the theta representations.

\end{abstract}

\maketitle

\section{Introduction}
Let $F$ be a number field containing a full set of $n$th roots of unity. Let $\ba$ be its adele ring. Let $\tgl_r(\ba)$ be a metaplectic $n$-fold cover of the general linear group. In their pioneering work, Kazhdan-Patterson \cite{KP1984} constructed generalized theta representations $\Theta_r$ on $\tgl_r(\ba)$ as multi-residues of Borel Eisenstein series. The local theta representations were  also constructed as the Langlands quotient of reducible principal series representations. They showed that (both globally and locally) the generalized theta representations are generic if and only if $n\geq r$; and uniqueness of Whittaker models holds if and only if $n=r$ or $r+1$ (when $n=r+1$, the uniqueness property only holds for certain covers). The theta representations and their unique models have been used to construct Rankin-Selberg integrals for symmetric power $L$-functions for the general linear groups; see Shimura \cite{Shimura1975}, Gelbart-Jacquet \cite{GJ1978}, Patterson-Piatetski-Shapiro \cite{PPS1989}, Bump-Ginzburg \cite{BG1992}, Bump-Ginzburg-Hoffstein \cite{BGH1996}, and Takeda \cite{Takeda2014}.

Suppose $r>n$. Motivated by the above background, one may ask the following natural questions:
\begin{enumerate}[(1)]
\item Does $\Theta_r$ support other types of Fourier coefficients?
\item If $\Theta_r$ supports a  nonzero Fourier coefficient, when does the uniqueness property hold?
\item If the uniqueness property holds for certain types of Fourier coefficients, can we use it to construct Rankin-Selberg integrals that represent Euler products?
\end{enumerate}

All the three questions have affirmative answers and this paper mainly addresses the first two questions. We first introduce a generalization of the Whittaker coefficients, which we call semi-Whittaker coefficients. Let $\lambda=(r_1 \cdots r_k)$ be a partition of $r$. Let $P_\lambda$ be the standard parabolic subgroup of $\GL_r$ whose Levi subgroup $M\cong \GL_{r_1}\times \cdots\times \GL_{r_k}$. Let $U_\lambda$ be its unipotent radical. Let $U$ be the standard unipotent subgroup of $\GL_r$. Fix a nontrivial additive character $\psi:F\bs\ba\to \bc^\times$. Let $\psi_\lambda:U(F)\bs U(\ba)\to \bc^\times$ be the  character  such that it acts as $\psi$ on the simple positive root subgroups contained in $M$, and acts trivially otherwise. The $\lambda$-semi-Whittaker coefficient of $\theta\in\Theta_r$ is defined to be
\[
\int\limits_{U(F)\bs U(\ba)}\theta(ug)\psi_\lambda(u) \ du.
\]
When the partition is $\lambda=(r)$, this recovers the usual Whittaker coefficients.

\begin{Thm}\label{thm:1.1}
\mbox{ }
\begin{enumerate}[{\normalfont (1)}]
\item
If there is an $r_i> n$, then
\[
\int\limits_{U(F)\bs U(\ba)}\theta(ug)\psi_\lambda(u) \ du
\]
is zero for all choices of data.

\item
If $r_i\leq n$ for all $i$, then
\[
\int\limits_{U(F)\bs U(\ba)}\theta(ug)\psi_\lambda(u) \ du
\]
is nonzero for some choice of data.

\item When $r=mn$, i.e. when the rank is a multiple of the degree, and the partition is $\lambda=(n^m)$, then global uniqueness of $\lambda$-semi-Whittaker models holds.
\end{enumerate}
\end{Thm}

We remark that the local version of the above theorem is also established (see Corollary \ref{cor:Vanishing semi whittaker}, \ref{cor:local nonvanishing}, and Theorem \ref{thm:multiplicity one semi whittaker}). Indeed, parts (1) and (3) are proved by using the local results, and part (2) is proved by using a  global argument. We also remark that when $n=2$ and $\lambda=(2^k)$ or $(2^k1)$ (depending on the parity of $r$), such coefficients and their uniqueness properties were already used in Bump and Ginzburg \cite{BG1992} in their work on symmetric square $L$-functions for $\GL(r)$.

The second type of Fourier coefficients we consider is the Fourier coefficients associated with unipotent orbits. The unipotent orbits of $\GL_r$ are parameterized by the partitions of $r$ via the Jordan decomposition. Given a unipotent orbit $\sco$, we can associate a set of Fourier coefficients; see Section \ref{sec:unipotent orbit} below. Roughly speaking, starting with a unipotent orbit $\sco$, we can define a unipotent subgroup $U_2(\sco)$. Let $\psi_{U_2(\sco)}: U_2(\sco)(F)\bs U_2(\sco)(\ba)\to\bc^\times$ be a character which is in general position. The Fourier coefficient of $\theta\in\Theta_r$ we want to consider is
\[
  \int\limits_{U_2(\sco)(F)\bs U_2(\sco)(\ba)}\theta(ug)\psi_{U_2(\sco)}(u)\ du.
\]
When the unipotent orbit is  $\sco=(r)$, this also recovers the usual Whittaker coefficients. There is a partial ordering  on the set of unipotent orbits. Our goal is to show that there is a unique maximal unipotent orbit    that supports  nonzero Fourier coefficients of $\Theta_r$ (see Definition \ref{def:unipotent orbit} below). Let $\sco(\Theta_r)$ be this orbit. The main results  for the Fourier coefficients associated with unipotent orbits are summarized as follows (Theorem \ref{thm:local unipotent orbit 2}, \ref{thm:global unipotent orbit}, and \ref{thm:multiplicity one unipotent orbit}).

\begin{Thm}\label{thm:1.2 unipotent orbit}
\begin{enumerate}[\normalfont(1)]
\item Write $r=an+b$ such that $a\in \bz_{\geq 0}$ and $0\leq b< n$. Then both locally and globally $\sco(\Theta_r)=(n^ab).$

\item
Let $v$ be a finite place such that $|n|_v=1$ and $\Theta_{r,v}$ is unramified. If $r=mn$ and $\sco=(n^m)$, then
\[
\dim\Hom_{U_2(\sco)(F_v)}(\Theta_{r,v},\psi_{U_2(\sco),v})=1.
\]
\end{enumerate}
\end{Thm}

This unique model is valuable and it already finds applications in Rankin-Selberg integrals for covering groups. In the research announcement by Friedberg, Ginzburg, Kaplan and the author \cite{CFGK2016},  the notion of  Whittaker-Speh-Shalika representation was introduced (see Definition \ref{def1}). Such representations are irreducible automorphic representations on $\tgl_r(\ba)$  and they possess unique functionals. The Whittaker-Speh-Shalika representations and their uniqueness models  are used in the  generalization of  the doubling methods to covering groups. The theta representations are  examples of such representations.

\begin{Thm}[Theorem \ref{thm:WSS}]
When $r=mn$,  $\Theta_r$ is a Whittaker-Speh-Shalika representation of type $(n,m)$.
\end{Thm}

This unique functional also plays a role in a new-way integral (Euler products with non-unique models) for covering groups; see Ginzburg \cite{Ginzburg2016}.

We now describe the ideas of the proofs. The proof of Theorem \ref{thm:1.1} is based on an induction in stages statement. We describe it in the global setup. Such an argument was also used in Bump-Friedberg-Ginzburg \cite{BFG2003} where they studied the Fourier coefficients of theta representations on the double covers of odd orthogonal groups. First of all, we can rewrite the $\lambda$-semi-Whittaker coefficients as
\[
\int\limits_{U(F)\bs U(\ba)}\theta(ug)\psi_\lambda(u) \ du=\int\limits_{U\cap M(F)\bs U\cap M(\ba)}\ \int\limits_{U_{\lambda}(F)\bs U_\lambda(\ba)}\theta(vug)\ dv \ \psi_\lambda(u) \ du.
\]
The inner integral is actually a constant term of the theta function. To compute it, we compute the constant term of the Eisenstein series and use the fact that the multi-residue operator and the constant term operator commute. By the standard unfolding argument,  the constant term of the Eisenstein series is a sum of Eisenstein series on $\tm(\ba)$. After applying the multi-residue operator, only one term survives. This implies that the constant term of a theta function is actually a ``theta function'' on $\tm(\ba).$ This fact is also called ``periodicity'' in \cite{KP1984} and \cite{BG1992}.

Now we are facing a difficulty which did not appear in \cite{BFG2003}. In the double cover of the odd orthogonal case, the constant terms of theta functions give rise to a representation on the cover of the Levi subgroup. In that case, different blocks commute in $\tm(\ba)$. Thus, one can take theta representations on each block and form the tensor product. It is shown that the tensor product of theta representations on each block is the same as the theta representation on $\tm(\ba)$. We would like to seek an analogous result for the general linear group. However, in the general linear case, when we restrict the metaplectic cover to  $\tm$, the blocks never commute (except when $n=2$). In fact, there is even no natural map between $\tm$ and $\tgl_{r_1}\times \cdots \times \tgl_{r_k}$. This means that, starting with representations on the $\tgl_{r_i}$, there is no direct way to construct a representation of $\tm$.

To overcome this difficulty, a construction called the metaplectic tensor product has been introduced (see Section \ref{sec:local metaplectic tensor} and \ref{sec:global meta tensor}). The local version is developed in Mezo \cite{Mezo2004} and the global version is given in  Takeda \cite{Takeda2016,Takeda2017}. Roughly speaking, the construction goes as follows (both locally and globally). Let $\tgln_{r_i}$ be the subgroup of $\tgl_{r_i}$, consisting of those elements whose determinants are $n$th powers. Let $\tmn$ be the subgroup of $\tm$ consisting of those elements such that the determinants of all the blocks are $n$th powers. The $\tgln_{r_i}$'s commute in $\tm$, and $\tmn$ is the direct product of   $\tgln_{r_1}, \cdots, \tgln_{r_k}$ with  amalgamated $\mu_n$.

Now start with representations $\pi_i$ on $\tgl_{r_i}$. We first restrict $\pi_i$ to $\tgln_{r_i}$, and pick an irreducible constituent $\pi_i^{(n)}$. Then we take the tensor product $\pi_1^{(n)}\otimes \cdots\otimes\pi_k^{(n)}$. This is a representation of $\tmn$. We then use induction  to obtain a representation of $\tm$. Extra care must be taken in order to establish the well-definedness and irreducibility of such constructions.

\begin{Thm}[Rough form]
Both locally and globally,
\[
\Theta_{\tm}\cong \Theta_{r_1}\tilde\otimes \cdots \tilde\otimes \Theta_{r_k}.
\]
\end{Thm}

The local version is given in Theorem \ref{thm:meteplectic tensor of theta}, and the global version is  Theorem \ref{thm:constant term 2}. Once we have the induction in stages statement, Theorem \ref{thm:1.1} can be established by carefully analyzing the restriction and induction process. In the local setup, we give an explicit formula for the dimension of the twisted Jacquet module $J_{U,\psi_\lambda}(\Theta_r)$.

Theorem \ref{thm:1.2 unipotent orbit} is proved in Sections \ref{sec:local unipotent} and \ref{sec:global unipotent}. The proof consists of two parts. The first part is to show that any unipotent orbit greater than or not comparable to $(n^ab)$ does not support any Fourier coefficients. The second part is to show that $(n^ab)$ actually supports a nonzero Fourier coefficient. The idea is to build a relation between the semi-Whittaker coefficients and the Fourier coefficients associated with unipotent orbits.  Once we know enough information about the semi-Whittaker coefficients, the unipotent orbit attached to the representation can be determined.

Two tools play crucial roles in the proof. The first one is called root exchange. This allows us to replace the domain of integration with a slightly different one. The second one is the Fourier expansion. This allows us to enlarge the domain of integration if we know certain coefficients vanish (this is usually related to the vanishing of semi-Whittaker coefficients). When we combine these tools in a systematic way, vanishing and nonvanishing of Fourier coefficients associated with unipotent orbits can be related to the results on the semi-Whittaker coefficients. Furthermore, when $n$ and $b$ have the same parity, we actually establish an identity between these coefficients. In particular, Theorem \ref{thm:1.2 unipotent orbit} part (2) follows from Theorem \ref{thm:1.1} part (3).

The remainder of this paper is organized as follows. Section \ref{sec:notation} introduces notations and defines metaplectic covers of general linear groups. Certain issues such as centers and maximal abelian subgroups are also discussed. The local theory of semi-Whittaker functionals is developed in Section \ref{sec:local theory}. We first review the principal series representations and theta representations of $\tgl_r(F_v)$. In Section \ref{sec:restrictions}, we give an explicit description of the restriction of these representations to $\tgln_{r}(F_v)$. These results are used to provide examples of the metaplectic tensor product in Section \ref{sec:examples}. We then carefully analyze the construction and compute the dimensions of some twisted Jacquet modules in Section \ref{sec:semiwhittaker}.
Section \ref{sec:semi whittaker global} is devoted to the global theory.  The nonvanishing part of Theorem \ref{thm:1.1} is proved in Theorem \ref{thm:Global Nonvanishing semi whittaker}.
In Section \ref{sec:unipotent orbit}, we review the association of Fourier coefficients to a unipotent orbit. The unipotent orbit attached to the theta representations is determined in Sections \ref{sec:local unipotent} and \ref{sec:global unipotent}. Section \ref{sec:local unipotent} introduces the local argument. The relation between semi-Whittaker coefficients and Fourier coefficients associated with unipotent orbit is established in a series of lemmas. Section \ref{sec:global unipotent} describes the corresponding global picture.

\subsection*{Acknowledgement}

The work in this paper forms part of my Boston College PhD thesis. I heartily thank my advisor, Solomon Friedberg, for his guidance and support. I am  also indebted to David Ginzburg for many helpful discussions and comments.  I would also like to thank the referee for careful reading and very detailed and helpful comments for an earlier version of the paper. This work was supported by the National Science Foundation, grant number 1500977.

\section{Notations and Preliminaries}

\subsection{Notations}\label{sec:notation}

Fix a positive integer $n$ and let
\[
\mu_n(F)=\{x\in F: x^n=1\}
\]
be the group of $n$th roots of unity in a field $F$. In this paper we always assume $|\mu_n(F)|=n$.  Fix, once and for all, an embedding $\epsilon:\mu_n\to\bc^\times$. We always write $\mu_n$ for $\mu_n(F)$, if there is no confusion. We often invoke the convention of omitting $\epsilon$ from the notation. All representations  which we consider are representations where  $\mu_n$ acts by scalars by the embedding $\epsilon$ . Such representations are called \textit{genuine}.

If $F$ is a non-Archimedean local field, we denote by $\fo$ the ring of integers of $F$. Let $\mathrm{val}$ be the normalized valuation on $F$.  Let $|\cdot|_F$ be the normalized absolute value on $F$. Let
\[
(\mbox{ },\mbox{ })=(\mbox{ },\mbox{ })_{F,n}:F^\times \times F^\times \to \mu_n(F)
\]
be the $n$th order Hilbert symbol. It is a bilinear form on $F^\times$ that defines a nondegenerate bilinear form on $F^\times/F^{\times n}$ and satisfies
\[
(x,-x)=(x,y)(y,x)=1,\qquad x,y\in F^\times.
\]

When $F$ is a number field, and $v$ is a place of $F$, we denote by $F_v$ the completion of $F$ at $v$. When $v$ is non-Archimedean, we let $\fo_v$ be the ring of integers of $F_v$.

For $\GL_r$, let $B=TU$ be the standard Borel subgroup with unipotent radical $U$ and maximal torus $T$. The set $\Phi=\{(i,j):1\leq i\neq j\leq r\}$ is identified with the set of roots of $\GL_r$ in the usual way. Let $\Phi^+$ denote the set of positive roots with respect to $B$.

For a partition $\lambda=(r_1\cdots r_k)$ of $r$,  let $P_\lambda$ be the standard parabolic subgroup of $\GL_r$ whose Levi part $M_\lambda$ is
$\GL_{r_1}\times \cdots \times \GL_{r_k}$ embedded diagonally
\[
(g_1,\cdots,g_k)\mapsto \mathrm{diag}(g_1,\cdots, g_k), \qquad g_i\in \GL_{r_i},
\]
and let $U_\lambda$ denote the unipotent radical of $P_\lam$. We usually write $M$ for $M_\lambda$  when the partition is fixed. We usually write $m\in M$ by $m=\mathrm{diag}(g_1,\cdots,g_k)$ with $g_i\in \GL_{r_i}$. Let $\Phi_\lambda$ and $\Phi_\lambda^+$ denote the set of roots and positive roots contained in $M_\lambda$, respectively. We also write $B_M=B\cap M$ and $U_M=U\cap M$. We sometimes add subscript $M$ to indicate the ambient group. For example, $T_M$ (which is $T$) is viewed as a maximal torus of $M$. We might still use $T$ for $T_M$ when there is no confusion.

Let $W$ be the set of all $r\times r $ permutation matrices. The Weyl group of $\GL_r$ is identified with $W$.  We also identify $W$ as the group of permutations of $\{1, 2, \cdots, r\}$  via
\[
w=(\delta_{i,w(j)}).
\]
The action of $W$ on $\Phi$ is given by $w(i,j)=(w(i),w(j))$. For a  Levi subgroup $M_\lambda$, let $W(M_\lambda)$ be the subset of permutation matrices contained in $M_\lambda$. The group $W(M_\lambda)$ is identified with the Weyl group of $M_\lambda$ (as sets). Let
\[
w^{M_\lambda}=\begin{pmatrix}
 & & & I_{r_1}\\
 & & I_{r_{2}} &\\
 & \reflectbox{$\ddots$}&&\\
 I_{r_k}&&&\\
\end{pmatrix}\in W.
\]
The element $w^{M_\lambda}$ sends $\GL_{r_k}\times \cdots\times \GL_{r_1}$ to $M_\lambda$.

For any group $G$ and elements $g,h\in G$, we define ${}^g h=ghg^{-1}$. For a subgroup $H\subset G$ and a representation $\pi$ of $H$, let ${}^g \pi$  be the representation of ${}^g H$ defined by ${}^g\pi(h')=\pi(g^{-1}h'g)$ for $h'\in {}^gH$.

Let $F$ be a local field. Let $\psi$ be a nontrivial additive character on $F$. In this paper we need to consider several characters on various  subgroups of $U$. We make the following convention. For a partition  $(p_1\cdots p_k)$ of $r'\leq r$, let $\Delta=\{i:1\leq i\leq r'\}\bs\{p_1, p_1+p_2,\cdots, p_1+\cdots+p_k\}$. Let $V$ be a subgroup of $U$ such that $V$ contains all the root subgroups associated to $\al=(i,i+1)$ for $i\in\Delta$. Let $\psi_{(p_1\cdots p_k)}:V\to \bc^\times$ be a character such that $\psi_{(p_1\cdots p_k)}$ acts as $\psi$ on the root subgroups associated to $\al=(i,i+1)$ for $i\in\Delta$, and acts trivially otherwise. Thus $\psi_{(r)}$ and $\psi_{(1^r)}$ are the usual Whittaker character and the trivial character on $U$, respectively. When $F$ is a number field and $\psi$ is a nontrivial additive character of $F\bs \ba$, these characters can be defined analogously.

Let $F$ be a non-Archimedean field. Let $\psi_V$ be a character on a unipotent subgroup $V$ of $U$. We use $J_V$  to denote the Jacquet functor with respect to $V$. The functor $J_{V,\psi_V}$ is the twisted Jacquet functor with respect to $(V,\psi_V)$.

Throughout the paper, the induction $\Ind$ is not normalized.

\subsection{The local metaplectic cover \texorpdfstring{$\tgl_r(F)$}{}}

Let $F$ be a local field of characteristic $0$ that contains all the $n$th roots of unity. Associated to every $2$-cocycle $\sigma:\GL_r(F)\times \GL_r(F)\to \mu_n(F)$, there is a central extension $\tgl_r(F)$ of $\GL_r(F)$ by $\mu_n$ satisfying an exact sequence
\[
1\to \mu_n\xrightarrow{\iota} \tgl_r(F)\xrightarrow{\textbf{p}} \GL_r(F)\to 1.
\]
We call $\tgl_r(F)$ a \textit{metaplectic $n$-fold cover of $\GL_r(F)$}. As a set, we can realize $\tgl_r(F)$ as
\[
\tgl_r(F)=\GL_r(F)\times \mu_n=\{( g,\zeta): g\in \GL_r(F),\zeta\in \mu_n\}.
\]
Notice that $\tgl_r(F)$ is not the $F$-rational points of an algebraic group, but this notation is standard. We use $\tgl_r$ to denote $\tgl_r(F)$. This abuse of notation is widely used in this paper, especially in the local setup. The embedding $\iota$ and the projection $\textbf{p}$ are given by
\[
\iota(\zeta)=(I_r,\zeta)\mbox{ and } \textbf{p}(g,\zeta)=g
\]
where $I_r$ is the identity element of $\GL_r$. The multiplication is defined in terms of $\sigma$ as follows,
\[
(g_1,\zeta_1)\cdot(g_2,\zeta_2)=(g_1g_2,\zeta_1\zeta_2\sigma(g_1,g_2)).
\]

For any subset $X\subset \GL_r$, let
\[
\widetilde{X}=\textbf{p}^{-1}(X)\subset \tgl_r.
\]
We also fix the section $\textbf{s}:\GL_r\to \tgl_r$ of $\textbf{p}$ given by $\textbf{s}(g)=(g,1)$. Then for $g_1,g_2\in \GL_r$,
\[
\textbf{s}(g_1)\textbf{s}(g_2)=(g_1g_2,\sigma(g_1,g_2)).
\]

In \cite{KP1984}, Kazhdan-Patterson provided $2$-cocycles $\sigma^{(c)}$ parameterized by $c\in\bz/n\bz$.  They are related by
\begin{equation}\label{eq:cocycle twist}
\sigma^{(c)}(g_1,g_2)=\sigma^{(0)}(g_1,g_2)(\det g_1,\det g_2)^c,\qquad g_1,g_2\in \GL_r.
\end{equation}
In this paper, we use the $2$-cocyles constructed in Banks-Levy-Sepanski \cite{BLS1999}.  The $2$-cocycles in \cite{BLS1999}  satisfy a block compatibility property.

Let $\sigma^{(0)}=\sigma_r^{(0)}$, and $\sigma^{(c)}=\sigma_r^{(c)}$ be related to $\sigma_r^{(0)}$ by Eq. (\ref{eq:cocycle twist}). Block compatibility means the following.
If $r=r_1+\cdots + r_k$ and $g_i,g'_i\in \GL_{r_i}$ for $i=1,\cdots, k$, then
\[
\begin{aligned}
\sigma_r^{(c)}(\text{diag}(g_1,\cdots, g_k),&\text{diag}(g'_1,\cdots,g'_k))\\
=&\left[\prod_{i=1}^k \sigma_{r_i}^{(c)}(g_i,g_i')\right]\cdot\left[\prod_{i<j} (\det g_i,\det g'_j)^{c+1} (\det g_j,\det g'_i)^c\right].
\end{aligned}
\]

Throughout the paper we fix the positive integers $r$ and $n$ and the modulus class $c\in \bz/n\bz$ and let $\sigma=\sigma_r^{(c)}$. Note that the restriction of $\sigma$ to $T$ is given by
\[
\sigma(t,t')=
\left[
\prod_{i<j}(t_i,t'_j)
\right]
\cdot
\prod_{i,j}(t_i,t'_j)^c
\]
for $t=\mathrm{diag}(t_1,\cdots,t_r)$ and $t'=\mathrm{diag}(t_1',\cdots,t_r')$.

The group $U$ splits in $\tgl_r$. In fact $\textbf{s}\mid_{U}$ is an embedding of $U$ in $\tgl_r$ (\cite{McNamara2012} Proposition 2).  Let $K=\GL_r(\fo)$. When $|n|_F=1$, $K$ also splits in $\tgl_r$ (\cite{McNamara2012} Theorem 2). There is a map $\kappa: K\to\mu_n$ such that $g\mapsto \kappa^\ast(g)=(g,\kappa(g))$ is a group homomorphism from $K$ to $\tgl_r$. We denote its image by $K^\ast$.  We shall fix $\kappa$ such that $\kappa^\ast$  is what Kazhdan-Patterson refer to as the canonical lift of $K$ to $\tgl_r$. It is characterized by the property that
\[
\textbf{s}|_{T\cap K}=\kappa^\ast|_{T\cap K}, \textbf{s}|_W
=\kappa^\ast|_{W}, \text{ and }\textbf{s}|_{U\cap K}=\kappa^\ast|_{U\cap K}.
\]
(\cite{KP1984} Proposition 0.1.3). The topology of $\tgl_r$ as a locally compact group is determined by this embedding.

\subsection{Centers}

The following lemma is Takeda \cite{Takeda2016} Lemma 3.13, which is very useful for us.
\begin{Lem}\label{lem:takeda 1}
Let $F$ be a local field. Then for each $g\in\GL_r$ and $a\in F^\times$,
\[
\sigma_r(g,aI_r)\sigma_r(aI_r,g)^{-1}=(\det(g),a^{r-1+2cr}).
\]
\end{Lem}

\begin{Lem}
Let $n_1=\gcd(n,2rc+r-1)$, and $n_2=\dfrac{n}{n_1}$. Then the center of $\tgl_r$ is
\[
\begin{aligned}
Z_{\tgl_r}=&\{(zI_r,\zeta): z^{2rc+r-1}\in F^{\times n}\}\\
=&\{(zI_r,\zeta):z\in F^{\times n_2}\}.
\end{aligned}
\]
\end{Lem}

The first part is proved in \cite{KP1984} Proposition 0.1.1, and the second part is proved in Chinta-Offen \cite{CO2013} Lemma 1.

The center of $\tto$ is also determined in \cite{KP1984}. Let $\tto^n=\{(t^n,\zeta):t\in T\}$.

\begin{Lem}
The center of $\tto$ is $Z_{\tgl_r}\tto^n$.
\end{Lem}

Let $\tgln_r:=\{g\in\tgl_r:\det g \in F^{\times n}\}$. We are interested in this group since it controls the representation theory of $\tgl_r$. Moreover, it plays a role in developing tensor products and parabolic inductions for metaplectic groups; see Section \ref{sec:local metaplectic tensor}.  Let $\tton:=\tgln_r\cap \tto$. The centers of $\tgln_r$ and $\tton$ behave better than the centers of $\tgl_r$ and $\tto$.

\begin{Lem}
The center of $\tgln_r$ is
\[
\begin{aligned}
Z_{\tgln_r}=\widetilde{Z}\cap \tgln_r=&\{(aI_r,\zeta): a^r\in F^{\times n}\}\\
=&\{(aI_r,\zeta): a\in F^{\times \frac{n}{\gcd(n,r)} }\}.
\end{aligned}
\]
\end{Lem}

\begin{proof}
The first equality is immediate from Lemma \ref{lem:takeda 1}. For the second equality, the proof is exactly the same as in \cite{CO2013} Lemma 1.
\end{proof}

The proof of the following lemma is also straightforward.

\begin{Lem}
The center of $\tton$ is $Z_{\tgln_r}\tto^n$.
\end{Lem}

\subsection{Maximal abelian groups}

Maximal abelian subgroups of $\tto$ play an important role in the representation theory of $\tto$. Let $\tto_\ast$ be a maximal abelian subgroup $\tto$.  In Section \ref{sec:restrictions}-\ref{sec:semiwhittaker},  we assume that $\tton\cap \tto_\ast$ is a maximal abelian subgroup of $\tton$, unless otherwise specified.

We briefly explain why such a group exists. First we start with a maximal abelian subgroup $\tton_\ast$ of $\tton$. If $\tton_\ast$ is not a maximal abelian subgroup of $\tto$, then we can choose $x\in \tto-\tton$ such that the group generated by $\tton_\ast$ and $x$ is abelian. Notice $|\tto/\tton|$ is finite. Thus we can repeat this process until we obtain a maximal abelian subgroup of $\tto$.

A concrete construction of maximal abelian groups  is given in Section \ref{sec:explicit maximal abelian}.

\subsection{The global metaplectic cover \texorpdfstring{$\tgl_r(\ba)$}{}}\label{sec:global metaplectic cover}

Let $F$ be a number field that contains all the $n$th roots of unity and $\ba$ be the ring of adeles. To construct a metaplectic $n$-fold cover of $\tgl_r(\ba)$ of $\GL_r(\ba)$, we follow \cite{Takeda2016} Section 2.2. The adelic $2$-cocycle $\tau_r$ is defined by
\[
\tau_r(g,g')=\prod_v \tau_{r,v}(g_v,g'_v),
\]
for $g,g' \in \GL_r(\ba)$. Here, the local cocycle is obtained from the block-compatible cocycle, multiplied by a suitable coboundary. It can be shown that there is a section $\textbf{s}_r:\GL_r(F)\to \tgl_r(\ba)$ such that $\GL_r(F)$ splits in $\tgl_r(\ba)$. The center $Z_{\tgl_r(\ba)}$ of $\tgl_r(\ba)$ can be easily found by using the local results. As in the local case, we define
\[
\tgln_r(\ba);=\{g\in\tgl_r(\ba): \det g \in\ba^{\times n}\}.
\]

The group $\tgl_r(\ba)$ can also be described as a quotient of a restricted direct product of the groups $\tgl_r(F_v)$. First we consider the restricted direct product $\prod_v'\tgl_r(F_v)$ with respect to $K_v$ for all $v$ with $v\nmid n$ and $v\nmid \infty$. Denote each element in this restricted direct product by $\prod_v'(g_v,\zeta_v)$ so that $g_v\in K_v$ and $\zeta_v=1$ for almost all $v$. Then
\begin{equation}\label{eq:global restricted tensor product}
\rho:\prod_v{}'\tgl_r(F_v)\to \tgl_r(\ba),\qquad \prod_v{}'(g_v,\zeta_v)\mapsto (\prod_v{}' g_v,\prod_v \zeta_v)
\end{equation}
is surjective group homomorphism. (Notice that $\prod_v \zeta_v$ is a finite product.) We have
\[
\prod_v{}'\tgl_r(F_v)/\mathrm{ker }\rho\simeq \tgl_r(\ba).
\]

Thus we have the notions of automorphic representations and automorphic forms on $\tgl_r(\ba)$. We now explain how to write an irreducible automorphic representation $\pi$ on $\tgl_r(\ba)$ as the ``metaplectic restricted tensor product'' $\widetilde\otimes_v{}' \pi_v$  in the sense of \cite{Takeda2017} Section 2. First of all, we view $\pi$ as a representation on the restricted direct product $\prod_v' \tgl_r(F_v)$ by pulling it back by $\rho$ in Eq. (\ref{eq:global restricted tensor product}). By the usual tensor product theorem for the restricted tensor product, we obtain $\pi\circ\rho\simeq \otimes_v'\pi_v$, where each $\pi_v$ is genuine. We call $\pi_v$ the \textit{irreducible constituent} of $\pi$ at $v$. For almost all $v$, $\pi_v$ is $K_v$-spherical.  The representation $\otimes_v'\pi_v$ descends to a representation $\tgl_r(\ba)$. Thus we write
\[
\pi\simeq \widetilde\otimes_v{}' \pi_v.
\]
Notice that the space of $\widetilde\otimes_v{}' \pi_v$ is the same as $\otimes_v{}' \pi_v$.

\subsection{Metaplectic cover of Levi subgroups}

Let $\lambda=(r_1\cdots r_k)$ be a partition of $r$. Let $M:=M_\lambda$ be the Levi subgroup of $\GL_r$ described in Section \ref{sec:notation}. This section discusses  metaplectic covers $\tm$, both locally and globally. The $2$-cocycle $\tau_r$ does not satisfy block-compatibility. To get round it, an equivalent cocycle $\tau_M$ was introduced in \cite{Takeda2016} Section 3. We use this cocycle to define $\tm$. Notice that the blocks $\GL_{r_i}$ do not commute with each other. Let $R=F$ if $F$ is local and $R=\ba$ is $F$ is global. Define
\[
\tmn(R)=\{(\mathrm{diag}(g_1,\cdots, g_k),\zeta):\det g_i\in R^{\times n}\}.
\]

Let $T_M$ be the maximal torus consisting of diagonal matrices. We write $T_i=T\cap \GL_{r_i}$, where $\GL_{r_i}$ is embedded in $\GL_r$ via
\[
g\mapsto \mathrm{diag}(I_{r_1},\cdots, g, \cdots, I_{r_k}).
\]
The torus $T_i$ can be viewed as a  maximal torus of $\GL_{r_i}$. Define $\tton_M=\tto_M\cap\tmn$. The  following results are proved in \cite{Takeda2016} Section 3. We omit the details.

\begin{Lem}
The center of $\tm(R)$ is
\[
Z_{\tm(R)}=\left\{
\begin{pmatrix}
a_1I_{r_1}& & \\
&\ddots&\\
&&a_k I_{r_k}
\end{pmatrix}:
a_i^{r-1+2cr}\in R^{\times n} \text{ and }a_1\equiv \cdots \equiv a_k\mod R^{\times n}
\right\}.
\]
\end{Lem}

\begin{Rem}
Notice that $Z_{\tgl_r}\tto^n=Z_{\tm}\tto_M^n$ and $Z_{\tgl_r}\tm^{(n)}=Z_{\tm}\tm^{(n)}$.
\end{Rem}

\begin{Lem}
The center of $\tmn$ is
\[
Z_{\tmn}=\left\{
\left(
\begin{pmatrix}
a_1I_{r_1}& & \\
&\ddots&\\
&&a_k I_{r_k}
\end{pmatrix},\zeta
\right):
a_i^{r_i}\in R^{\times n}
\right\}.
\]
\end{Lem}

\begin{Lem}
The center of $\tton_M$ is $Z_{\tmn}\tto^n.$
\end{Lem}

Let $\tto_{M,\ast}$ be a  maximal abelian subgroup of $\tto_M$. We again assume $\tto_{M,\ast}\cap \tton_M$ is a maximal abelian subgroup of $\tton_M$.

One can also talk about automorphic forms and representations on $\tm(\ba)$ as well.

\section{Local Theory}\label{sec:local theory}

In this section, $F$ is a non-Archimedean local field. Recall that we use $\tgl_r$ to denote $\tgl_r(F)$.

\subsection{The principal series representations}

The principal series representations of $\tgl_r$ were studied in \cite{KP1984}. For the generalization to metaplectic covers of  other reductive groups, see McNamara \cite{McNamara2012}.

We start with the representation theory of $\tto$. In general, $\tto$ is not abelian, but it is a two-step nilpotent group.  The irreducible genuine representations of $\tto$ are parameterized in the following way (\cite{McNamara2012} Theorem 3): start with a genuine character $\chi$ on the center of $\tto$ and extend it to a character $\chi'$ on any maximal abelian subgroup $\tto_\ast$, then the induced representation $i(\chi'):=\Ind_{\tto_\ast}^{\tto}\chi'$ is irreducible (see \cite{McNamara2012} Theorem 3). This construction is  independent of the choice of $\tto_\ast$ and of the extension of characters.

We extend $i(\chi')$ to a representation $i_{\tb}(\chi')$ on $\tb=\tto U$ by letting $U$ act trivially. Let $\delta_B$ be the modular quasicharacter of $B$. Then $\Ind_{\tb}^{\tgl_r}i_{\tb}(\chi')\delta_B^{1/2}$ is the principal series representation.   This representation is denoted by $I(\chi')$, although its isomorphism class only depends on $\chi$.

There is an alternative way to describe the principal series representations.   We can extend the character $\chi'$ to $\tb_\ast=\tto_\ast U$, and then induce it to $\tgl_r$.  The representation $\Ind_{\tb_\ast}^{\tgl_r}\chi'\delta_B^{1/2}$ is isomorphic to $I(\chi')$.

The representation $I(\chi')$ is irreducible when $\chi$ is in general position. For a positive root $\alpha$, there is an embedding
$i_\alpha:\SL_2\to \GL_r$.
Define
\[
\chi^n_\alpha(t)=\chi\left(i_\alpha \begin{pmatrix}t & \\ & t^{-1}\end{pmatrix}^n\right).
\]

\begin{Thm}
Suppose that $\chi_\alpha^n\neq |\cdot|_F^{\pm 1}$ for all the positive roots $\alpha$. Then $I(\chi')$ is irreducible.
\end{Thm}

This is proved by the theory of intertwining operators; see \cite{KP1984} Corollary I.2.8.

If $\chi_\alpha^n=|\cdot|_F$ for all the positive simple roots $\alpha$, we call $\chi$ \textit{exceptional}. In this case, $I(\chi')$ is reducible, and we are interested in the unique irreducible subquotient of $I(\chi')$. Recall that the intertwining operator  $T_w:I(\chi')\to I({}^w\chi')$ is defined as
\[
(T_wf)(g)=\int\limits_{U(w)} f(w^{-1}ug)du.
\]
where $U(w)$ is the subgroup of $U$ corresponding to roots $\alpha>0$ such that $w^{-1}\alpha<0$. If this converges for all $f\in I(\chi')$ and is non-trivial, then it is a generator of $\Hom_{\tgl_r}(I(\chi'),I({}^w \chi'))$. For general $\chi$, the intertwining operator can be defined via analytic continuation.

\begin{Thm}
Let $\chi$ be exceptional. Let
\[
\Theta(\chi')=\mathrm{Im}(T_{w_0}:I(\chi')\to I({}^{w_0}\chi')),
\]
where $w_0$ is the longest element of $W$. Then
\begin{enumerate}[\normalfont (1)]
\item $\Theta(\chi')$ is the unique irreducible subrepresentation of $I({}^{w_0}\chi')$.
\item $\Theta(\chi')$ is the unique irreducible quotient representation of $I(\chi')$.
\item The Jacquet module  $J_U(\Theta(\chi'))\cong \ind_{{}^{w_0} \tto_\ast}^{\tto}({}^{w_0}\chi'\delta_B^{1/2})$.
\end{enumerate}
\end{Thm}

This is \cite{KP1984} Theorem I.2.9. $\Theta(\chi')$ is called exceptional.

The Whittaker models of exceptional representations are studied in \cite{KP1984} Section I.3. These authors have shown the following results.

\begin{Prop}\label{prop:KP1984}
Suppose that  $|n|_F=1$.
\begin{enumerate}[\normalfont (1)]
\item The representation $\Theta(\chi')$ has a unique Whittaker model if and only if $n=r$, or otherwise $n=r+1$, and $2(c+1)\equiv 0\mod n$.
\item The representation $\Theta(\chi')$ does not have a Whittaker model if $n\leq r-1$.
\item The representation $\Theta(\chi')$ has a finite number of independent nonzero Whittaker models if $n\geq r+1$.
\end{enumerate}
\end{Prop}

\begin{Rem}
In the above proposition, parts (1) and (3) are also true when $|n|_F\neq1$. This is shown in \cite{KP1984} Section II by using global arguments. Part (2) is expected to be true when $|n|_F\neq 1$, but this is known only when $n=2$; see Kaplan \cite{Kaplan} Theorem 2.6 and Flicker-Kazhdan-Savin \cite{FKS1990}.
\end{Rem}

\begin{Rem}
When $r=1$, we take $\Theta(\chi')$ to be $\Ind_{\tto_\ast}^{\tto}\chi'$. This fits into the metaplectic tensor product perfectly.
\end{Rem}

\subsection{Restrictions}\label{sec:restrictions}
We study the restriction functor $\mathrm{Res}_{\tgln_r}^{\tgl_r}$ in this section. We obtain an explicit description of the restriction of the principal series representations and exceptional representations from $\tgl_r$ to $\tgln_r$. This is useful in Section \ref{sec:examples} where we give explicit examples of the metaplectic tensor product.

Notice that $\tgl_r^{(n)}$ is an open normal subgroup of $\tgl_r$, and $\tgl_r/\tgl_r^{(n)}\cong F^\times/F^{\times n}$ is finite and abelian. By Gelbart-Knapp \cite{GK1982} Lemma 2.1,  if $I(\chi')$ is irreducible, and $\pi$ is an irreducible constituent of $I(\chi')|_{\tgl_r^{(n)}}$, then
\[
I(\chi')|_{\tgl_r^{(n)}}=\sum_{g} m \ {}^g\pi.
\]
The multiplicities $m$ depend only on $I(\chi')$, and the direct sum is over certain elements of $\tgl_r$.

From now on, we assume $\tton_\ast:=\tto_\ast\cap \tton$ is a maximal abelian subgroup of $\tton$. Let $\tb_\ast=\tto_\ast U$ and $\tb_\ast^{\square}=\tton_\ast U$.
\begin{Prop}\label{prop:decomposition of induced}
\begin{equation}\label{eq:decomposition}
I(\chi')|_{\tgln_r}\cong
\bigoplus_{x^{-1}\in \tto_\ast\bs \tto/\tton} \Ind_{{}^x\tb_\ast^{\square}}^{\tgln_r}({}^{x}\chi' \delta_B^{1/2})|_{{}^x\tb_\ast^{\square}}.
\end{equation}
\end{Prop}

\begin{proof}
This follows from Bernstein-Zelevinsky \cite{BZ1977} Theorem 5.2. We are working with representations of $\tgl_r$. Let us choose triples $\tb, \tto, U$ with trivial character on $U$ on the induced functor side, and $\tgl_{r}^{(n)},\tgl_r^{(n)},\{1\}$ with trivial character on $\{1\}$ on the Jacquet functor side. The Jacquet functor in this case is the restriction functor.

The resulting functor is glued by functors indexed by the double coset space $\tb\backslash \tgl_r/\tgln_r$. This double coset space is a singleton since $\tto\tgl_r^{(n)}=\tgl_r$. Therefore, the functor is the composition of the induction functor from $\tto\cap \tgl_r^{(n)}$ to $\tgl_r^{(n)}$ and the restriction functor $\mathrm{Res}_{\tton}^{\tto}$.

By \cite{GK1982} Lemma 2.1, $\ind_{\tto_\ast}^{\tto}\chi'|_{\tton}$ is a direct sum of irreducible $\tton$-representations. On the other hand,  it has a  Jordan-Holder series whose composition factors are
\[
\ind_{{}^x\tton_\ast}^{\tton} {}^{x}\chi',\qquad x^{-1}\in \tto_\ast\bs \tto/\tton.
\]
Notice that $\tton$ is also a Heisenberg group and $\tton\cap {}^x\tto_\ast={}^x(\tton_\ast)$ is again a maximal abelian subgroup of $\tton$. This implies $\ind_{{}^x\tton_\ast}^{\tton} {}^{x}\chi'$ is irreducible. Thus,
\[
(\ind_{\tto_\ast}^{\tto}\chi')|_{\tton}=\bigoplus_{x^{-1}\in \tto_\ast\bs \tto/\tton} \ind_{{}^x\tton_\ast}^{\tton} {}^{x}\chi'.
\]
Now the  proposition follows.
\end{proof}

\begin{Rem}
Notice that Eq. (\ref{eq:decomposition}) depends on the choice of maximal abelian subgroup. Indeed, when $\chi$ is in general position, the condition that $\tto_\ast\cap \tton=\tton_\ast$ implies each component is irreducible. Without this condition, we get a similar decomposition, but the components are reducible.
\end{Rem}

Next we show that, when $\chi$ is in general position, the components in Proposition \ref{prop:decomposition of induced}  are irreducible. Let us write $V({}^x\chi')=\Ind_{{}^x\tb_\ast^{\square}}^{\tgln_r}({}^{x}\chi' \delta_B^{1/2})|_{{}^x\tb_\ast^{\square}}$ for $x^{-1}\in \tto_\ast\bs \tto/\tton$. Thus Proposition \ref{prop:decomposition of induced} becomes
\[
I(\chi')|_{\tgln_r}\cong \bigoplus_{x^{-1}\in \tto_\ast\bs \tto/\tton} V({}^{x}\chi').
\]

\begin{Def}
A character of $Z_{\tgl_r} \tto^n$ or $Z_{\tgln_r} \tto^n$ is called regular if ${}^w \chi\neq \chi$ for all $w\in W,w\neq I$.
\end{Def}

\begin{Lem}\label{lem:jacquet functor of induced}
\hspace{2em}
\begin{enumerate}[\normalfont (1)]
\item
The $\tton$-module $J_U(V(\chi'))$ has a Jordan-Holder series whose composition factors are
\[
\ind_{{}^w\tton_\ast }^{\tton} ({}^w\chi'\delta_B^{1/2}) ~ (w\in W).
\]
\item
If $\chi$ is regular, then for any extension $\chi'$, $\chi'|_{Z_{\tgln_r} \tto^n}$ is regular. Moreover,
\[
J_U(V(\chi'))\cong\bigoplus_{w\in W}\ind_{{}^w\tton_\ast }^{\tton}({}^w\chi'\delta_B^{1/2}).
\]
\end{enumerate}
\end{Lem}

\begin{proof}

The first part follows from \cite{BZ1977} Theorem 5.2. For the second part, we only need to show that $\chi'|_{Z_{\tgln_r} \tto^n}$ is regular. Indeed, if $\chi$ is regular, then for any $w\in W$, there exists $x\in Z_{\tgl_r}\tto^n$ such that $\chi(w^{-1}xw)\neq\chi(x)$. Without loss of generality, we may assume $x\in \tto^n$. This implies that $\chi'|_{Z_{\tgln_r}\tto^n}$ is regular for any extension $\chi'$ of $\chi$.
\end{proof}

\begin{Lem}
Let $\chi_1,\chi_2$ be two quasicharacters of $Z_{\tgl_r^{(n)}}\tto^n$ and let $\chi_1',\chi_2'$ be extensions to $\tton_\ast$. Suppose $\chi_1$ is regular. Then
\[
\dim\Hom_{\tgln_r}(V(\chi_1'),V(\chi_2'))\leq 1.
\]
The equality holds if and only if $\chi_2={}^w\chi_1$ for some $w\in W$.
\end{Lem}

\begin{proof}
This is an immediate application of Lemma \ref{lem:jacquet functor of induced}, Frobenuis reciprocity, and the fact that $\ind_{{}^w\tton_\ast }^{\tton} ({}^w\chi'\delta_B^{1/2})$ is irreducible.
\end{proof}

\begin{Lem}
The restriction of  the intertwining operator $T_w:I(\chi')\to I({}^w\chi')$ to Eq. (\ref{eq:decomposition}) gives
\[
T_w: V({}^x\chi')\to V({}^{wx} \chi').
\]
\end{Lem}

\begin{proof}

Recall that $I(\chi')$ is the space of smooth functions
\[
I(\chi')=\{f:\tgl_r\to \bc: f \text{ is smooth and } f(bg)=\chi'(b)\delta_B(b)^{1/2} f(g) \text{ for }b\in \widetilde{B}_\ast\}.
\]

The embedding of $V({}^x \chi')$ into $I(\chi')$ is given as follows. Let $f\in V({}^x \chi')$. Define
\[
\tilde f(g)=\left\{
\begin{aligned}
&\tilde{f}(xg) &\text{ if }x\in\tgln_r,\\
&0 &\text{ otherwise}. \\
\end{aligned}
\right.
\]
Then it is straightforward to check that $\tilde f\in I(\chi')$.

Now given $f\in V({}^x\chi')$. We can see that $T_w(\tilde f)$ is in the image of $V({}^{wx}\chi')$ in $I({}^w\chi')$.

\end{proof}

\begin{Prop}
If $\chi_\alpha^n\neq |\cdot|^{\pm 1}$ for all positive roots $\alpha$, then $V(\chi')$ is irreducible.
\end{Prop}

\begin{proof}
Under the assumption, $T_w:I(\chi')\to I({}^w\chi')$ is an isomorphism, and hence its restriction
\[
T_w: V(\chi')\to V({}^w \chi')
\]
is again an isomorphism. Arguing as in \cite{KP1984} Corollary I.2.8, we can show that $V(\chi')$ is irreducible.
\end{proof}

Similarly we can deduce results for exceptional representations.

\begin{Thm}
Let $\chi$ be exceptional. Let
\[
V_0(\chi')=\mathrm{Im}(T_{w_0}:V(\chi')\to V({}^{w_0}\chi')),
\]
where $w_0$ is the longest elements of $W$. Then
\begin{enumerate}[\normalfont (1)]
\item $V_0(\chi')$ is the unique irreducible subrepresentation of $V({}^{w_0}\chi')$.
\item $V_0(\chi')$ is the unique irreducible quotient representation of $V(\chi')$.
\item $J_U(V_0(\chi'))\cong \ind_{{}^{w_0}\tton_\ast}^{\tton}({}^{w_0}\chi'\delta_B^{1/2})$.
\end{enumerate}
\end{Thm}

\begin{proof}
The map
\[
T_{w_0}:I(\chi')\to I({}^{w_0}\chi')
\]
restricts to
\[
T_{w_0}:\bigoplus_{x^{-1}\in\tto_\ast\bs \tto/\tton} V({}^{x}\chi')\to \bigoplus_{x^{-1}\in\tto_\ast\bs \tto/\tton} V({}^{w_0x}\chi').
\]
This implies that
\[
\Theta(\chi')|_{\tgln_r}=\bigoplus_{x^{-1}\in \tto_\ast\bs \tto/\tton} V_0({}^{x}\chi').
\]

We first show part (3). From the exactness of the Jacquet functor, $J_U(V_0(\chi'))$ is a subrepresentation of both $J_U(V(\chi'))$ and $J_U(\Theta(\chi'))$. Therefore, $J_U(V_0(\chi')) \cong \ind_{{}^{w_0} \tton_\ast }^{\tton}({}^{w_0}\chi'\delta_B^{1/2})$.

The representation $\Theta(\chi')|_{\tgln_r}$ is a direct sum of irreducible constituents, which are conjugate to each other. Thus $V_0(\chi')$ is a direct sum of some of these components. This implies that $J_U(V_0(\chi'))$ is also a direct sum of the corresponding Jacquet modules which are conjugate to each other. Thus $V_0(\chi')$ is irreducible since $J_U(V_0(\chi'))$ is irreducible.

If $\pi$ is another irreducible quotient representation of $V(\chi')$, then its Jacquet module is a quotient of $J_U(V(\chi'))$, and hence there is a nonzero homomorphism $J_U(\pi)\to \ind^{\tton}_{{}^w \tton_\ast }({}^w \chi'\delta_B^{1/2})$ for some $w\in W$. By Frobenius reciprocity, there is a nonzero intertwining map $\pi\to V({}^w \chi')$. The composition
\[
V(\chi')\to \pi \to V({}^w \chi')
\]
is nonzero and it must be a constant multiple of $T_w$. Therefore, the composition
\[
V(\chi')\to \pi\to V({}^w \chi')\xrightarrow{T_{w_o w^{-1}}} V({}^{w_0}\chi')
\]
is $T_{w_0}$ and its image is $V_0(\chi')$. We see that $V_0(\chi')$ is a quotient of $\pi$, and since $\pi$ is irreducible, they must be the same. This proves part (2). Part (1) follows from part (2) by duality.

\end{proof}

As a corollary, we describe the decomposition of $\Theta(\chi')$ when restricted to $\tgln_r$.

\begin{Cor}\label{cor:restriction of theta}
\[
\Theta(\chi')|_{\tgln_r}\cong \bigoplus_{x^{-1}\in \tto_\ast\bs \tto/\tton} V_0({}^x\chi').
\]
\end{Cor}

\subsection{Principal series of Levi subgroups}

Let $\lambda$ be a partition of $r$ and write $\tm$ for $\tm_\lambda$. Recall $B_M=B\cap M$, and $U_M=U\cap M$. The principal series representations and exceptional representations can be similarly defined on $\tm$.
Recall we may identify $\GL_{r_i}$ as a subgroup of $M$ via the embedding
\[
g_i\mapsto \mathrm{diag}(I_{r_1},\cdots, g_i,\cdots, I_{r_k}).
\]
Let $B_i$ be the standard Borel subgroup of $\GL_{r_i}$ and $\delta_{B_i}$ be the modular quasicharacter of $B_i$ in $\GL_{r_i}$.

Let $\chi$ be a genuine character of $Z_{\tgl_r}\tto_M^n$ and $\chi'$ be a character of $\tto_{M,\ast}$ extending  $\chi$.  The genuine representation $\pi_{\tto_M}(\chi'):=\ind_{\tto_{M,\ast}}^{\tto_M}\chi'$ is irreducible. The principal series representation $I(\chi')$ is the induced representation $\Ind_{\widetilde{B_M}}^{\tm}\pi_{\tto}(\chi')\otimes \delta_M^{1/2}$, where $\delta_M=\delta_{B_1}\otimes \cdots\otimes\delta_{B_k}$. There is an alternative way to describe it as in the general linear case.

The theory of intertwining operators applies just as the general linear case. Therefore, $I(\chi')$ is irreducible when $\chi$ is in general position.

\begin{Thm}
Suppose that $\chi_\alpha^n\neq |\cdot|_F^{\pm 1}$ for all the positive roots $\alpha$ in $M$. Then $I(\chi')$ is irreducible.
\end{Thm}

If $\chi_\alpha^n=|\cdot|_F$ for all positive simple roots $\alpha$ in $M$, we call it \textit{exceptional}.
\begin{Thm}
Let $\chi$ be exceptional. Let
\[
\Theta(\chi')=\mathrm{Im}(T_{w_{M,0}}:I(\chi')\to I({}^{w_{M,0}}\chi')),
\]
where $w_{M,0}$ is the longest element of $W(M)$. Then
\begin{enumerate}[\normalfont(1)]
\item $\Theta(\chi')$ is the unique irreducible subrepresentation of $I({}^{w_{M,0}}\chi')$.
\item $\Theta(\chi')$ is the unique irreducible quotient representation of $I(\chi')$.
\item $J_{U_M}(\Theta(\chi'))\cong \ind_{w_{M,0} \tto_{M,\ast} w_{M,0}^{-1}}^{\tto_M}({}^{w_{M,0}}\chi'\delta_M^{1/2})$.
\end{enumerate}
\end{Thm}

We also want to study $I(\chi')|_{Z_{\tgl_r}\tmn}$, and $\Theta(\chi')|_{Z_{\tgl_r}\tmn}$. The  arguments in Section \ref{sec:restrictions} apply in this case without essential change. We only state the results here.
\begin{Prop}
\[
I(\chi')|_{Z_{\tgl_r}\tmn}\cong \bigoplus_{x^{-1}\in \tto_{M,\ast}\bs\tto_M/Z_{\tgl_r}\tton_M} \Ind_{{}^x(Z_{\tgl_r}\tmn\cap \tb_{M,\ast})}^{Z_{\tgl_r}\tmn}{}^x\chi'\delta_M^{1/2}.
\]
\end{Prop}

\begin{Prop}
If $\chi_\alpha^n\neq |\cdot|^{\pm1}$ for all positive roots $\alpha$ in $M$, then
\[
\Ind_{{}^x(Z_{\tgl_r}\tmn\cap \tb_{M,\ast})}^{Z_{\tgl_r}\tmn} {}^x\chi'\delta_M^{1/2}
\]
 is irreducible.
\end{Prop}

As in the general linear case, write $V({}^x\chi')=\Ind_{{}^x(Z_{\tgl_r}\tmn\cap \tb_{M,\ast})}^{Z_{\tgl_r}\tmn} {}^x\chi'\delta_M^{1/2}$.

\begin{Prop}
Let $\chi$ be exceptional. Let
\[
V_0(\chi')=\mathrm{Im}(T_{w_{M,0}}:V(\chi')\to V({}^{w_{M,0}}\chi')),
\]
where $w_{M,0}$ is the longest elements of $W(M)$. Then
\begin{enumerate}[\normalfont(1)]
\item $V_0(\chi')$ is the unique irreducible subrepresentation of $V({}^{w_{M,0}}\chi')$.
\item $V_0(\chi')$ is the unique irreducible quotient representation of $V(\chi')$.
\item $J_{U_M}(V_0(\chi'))\cong \ind_{Z_{\tgl_r}w_{M,0}\tton_{M,\ast} w_{M,0}^{-1}}^{Z_{\tgl_r}\tton_M}({}^{w_{M,0}}\chi'\delta_M^{1/2})$.
\end{enumerate}
\end{Prop}

\begin{Prop}
\[
\Theta(\chi')|_{Z_{\tgl_r}\tmn}\cong \bigoplus_{x^{-1}\in \tto_{M,\ast}\bs\tto_M/Z_{\tgl_r}\tton_M} V_0({}^x\chi').
\]
\end{Prop}

Lastly, let $\chi'$ be an exceptional character for $\tgl_r$. Let $P$ be the parabolic subgroup of $\GL_r$ with Levi subgroup $M$, and $R$ be its unipotent radical. Let $\delta_P$ be the modular quasicharacter of $\GL_r$ with respect to $P$. Recall we have $\delta_M\cdot\delta_P=\delta_{\GL_r}$ and $w_0=w_{M,0}w^M$.

\begin{Prop}
The character ${}^{w^M}\chi'\cdot\delta_P^{1/2}$ is exceptional  for $M$, and
\[
J_R(\Theta_{\tgl_r}(\chi'))\cong\Theta_{\tm}({}^{w^M}\chi'\cdot\delta_P^{1/2}).
\]
\end{Prop}

\begin{proof}

The Weyl element $w^M$ permutes blocks of $M$, and thus the character ${}^{w^M}\chi'\cdot\delta_P^{1/2}$ is exceptional for $\tm$. To prove the isomorphism of  Jacquet modules, we apply $J_{U_M}(-)$ on both sides. The left-hand side is
\[
J_{U_M}(J_R(\Theta_{\tgl_r}(\chi')))=J_U(\Theta_{\tgl_r}(\chi'))\cong\ind_{{}^{w_0} \tto_\ast }^{\tto}({}^{w_0}\chi'\delta_{\GL_r}^{1/2});
\]
while the right-hand side is
\[
J_{U_M}(\Theta_{\tm}({}^{w^M}\chi'\cdot\delta_P^{1/2}))\cong\ind_{w_{M,0}\tto_\ast w_{M,0}^{-1}}{}^{w_0}\chi'\cdot\delta_P^{1/2}\delta_M^{1/2}\cong\ind_{{}^{w_0} \tto_\ast }^{\tto}({}^{w_0}\chi'\delta_{\GL_r}^{1/2}).
\]
This implies that $J_R(\Theta_{\tgl_r}(\chi'))$ and $\Theta_{\tm}({}^{w^M}\chi'\cdot\delta_P^{1/2})$ are both irreducible subrepresentations of $I({}^{w^M}\chi'\cdot\delta_P^{1/2})$. Thus they are isomorphic.

\end{proof}

\subsection{The metaplectic tensor product}\label{sec:local metaplectic tensor}

One of the basic constructions in the representation theory of $\GL_r(F)$ is parabolic induction. Let $r=r_1+\cdots+r_k$ be a partition of $r$, and let $M=\GL_{r_1}\times\cdots\times \GL_{r_k}$ be a Levi subgroup. We start with a list of representations, one for each of $\GL_{r_1},\cdots,\GL_{r_n}$, and then form their tensor product to obtain a representation of $M$.  However, since $\tm$ is not simply the amalgamated direct product of the various $\widetilde{\GL}_{r_i}$, this construction cannot be generalized directly to the metaplectic case. Fortunately, we have a replacement, which is defined in Mezo \cite{Mezo2004}. We review the construction in this section. The two-fold cover case was outlined in Bump and Ginzburg \cite{BG1992}, and studied in full detail in Kable \cite{Kable2001}. For the global setup and further properties see Takeda \cite{Takeda2016,Takeda2017}.

We observe that any element $m\in \widetilde{M}$ may be written as  $\mathrm{diag}(g_1,\cdots, g_k)$, such that $\textbf{p}(g_i)\in \GL_{r_i}$ for $1\leq i\leq k$. Recall
\[
\widetilde{M}^{(n)}=\{m\in\widetilde{M}: \det g_1,\cdots,\det g_k\in F^{\times n}\}
\]
and $\tgl_{r_i}^{(n)}=\tm^{(n)}\cap \tgl_{r_i}$.

Let $\pi_1,\cdots,\pi_k$ be irreducible genuine representations of $\tgl_{r_1},\cdots,\tgl_{r_k}$, respectively. The construction of the metaplectic tensor product takes several steps.

First of all, for each $i$, fix an irreducible constituent $\pi_i^{(n)}$ of the restriction $\pi_i\mid_{\tgl_{r_i}^{(n)}}$ of $\pi_i$ to $\tgl_{r_i}^{(n)}$. Then we have
\[
\pi_i|_{\tgl_{r_i}^{(n)}}=\sum_g m_i \mbox{ }^g(\pi_i^{(n)}),
\]
where $g$ runs through a finite subset of $\tgl_{r_i}$, $m_i$ is a positive multiplicity and ${}^g(\pi_i^{(n)})$ is the representation twisted by $g$. Then we construct the tensor product representation
\[
\pi_1^{(n)}\otimes \cdots \otimes \pi_k^{(n)}
\]
of the group $\tgl_{r_1}^{(n)}\otimes \cdots\otimes \tgl_{r_k}^{(n)}$. Because the representations $\pi_1,\cdots,\pi_k$ are genuine, this tensor product representation  descends to a representation of the group $\tm^{(n)}$, i.e. the representation factors through the natural surjection
\[
\tgl_{r_1}^{(n)}\times \cdots \times\tgl_{r_k}^{(n)}\twoheadrightarrow \tm^{(n)}.
\]
We denote this representation of $\tm^{(n)}$ by
\[
\pi^{(n)}:=\pi_1^{(n)}\tilde\otimes\cdots \tilde\otimes \pi_k^{(n)},
\]
and call it the metaplectic tensor product of $\pi_1^{(n)},\cdots, \pi_k^{(n)}$.

Let $\omega$ be a character on the center $Z_{\tgl_r}$ such that for all $(aI_r,\zeta)\in Z_{\tgl(r)}\cap \tm^{(n)}$ where $a\in F^\times$, we have
\[
\omega(aI_r,\zeta)=\pi^{(n)}(aI_r,\zeta)=\zeta \pi_1^{(n)}(aI_{r_1},1)\cdots \pi_r^{(n)}(aI_{r_k},1).
\]
Namely, $\omega$ agrees with $\pi^{(n)}$ on the intersection $Z_{\tgl_{r}}\cap \tm^{(n)}$. We can extend $\pi^{(n)}$ to the representation
\[
\pi_\omega^{(n)}:=\omega\pi^{(n)}
\]
of $Z_{\tgl_r}\tm^{(n)}$ by letting $Z_{\tgl_r}$ act by $\omega$.

The last step is crucial. If we induce $\pi_\omega^{(n)}$ to $\tm$, the resulting representation is usually reducible. To get an irreducible representation, we extend the representation $\pi_\omega^{(n)}$ to a representation $\rho_\omega$ of a subgroup $\widetilde{H}$ of $\tm$ so that $\rho_\omega$ satisfies Mackey's irreducibility criterion and the induced representation
\[
\pi_\omega:=\Ind_{\widetilde{H}}^{\tm} \rho_\omega
\]
is irreducible. It is always possible to find such $\widetilde{H}$ and moreover $\widetilde{H}$ can be chosen to be normal. The construction of $\pi_\omega$ is independent of the choices of $\pi_i^{(n)}$, $\widetilde{H}$ and $\rho_\omega$, and it only depends on $\omega$ (see \cite{Mezo2004} Section 4).

We write
\[
\pi_\omega=(\pi_1\tilde\otimes\cdots\tilde\otimes \pi_k)_\omega
\]
and call it the metaplectic tensor product of $\pi_1,\cdots,\pi_k$ with the character $\omega$.

The metaplectic tensor product $\pi_\omega$ is unique up to twist.
\begin{Prop}[\cite{Mezo2004} Lemma 5.1]\label{prop:Mezo}
Let
\[
\pi_1,\cdots,\pi_k\qquad\text{ and }\qquad \pi_1',\cdots, \pi_k'
\]
be genuine representations of $\tgl_{r_1},\cdots,\tgl_{r_k}$. They give rise to isomorphic metaplectic tensor products with a character $\omega$, i.e.
\[
(\pi_1\tilde\otimes \cdots \tilde\otimes \pi_k)_\omega\cong (\pi_1'\tilde\otimes\cdots\tilde\otimes \pi_k')_\omega
\]
if and only for each $i$ there exists a character $\omega_i$ of $\tgl_{r_i}$, trivial on $\tgl_{r_i}^{(n)}$, such that $\pi_i\cong \omega_i\otimes \pi_i'$.
\end{Prop}

\begin{Rem}
Notice that the metaplectic tensor product generally depends on the choice of $\omega$. If the center $Z_{\tgl_r}$ is already contained in $\tm^{(n)}$, we have $\pi_\omega^{(n)}=\pi^{(n)}$ and hence there is no actual choice for $\omega$ and the metaplectic tensor product is canonical. This is the case, for example, when $n=2$ or $n\mid r$.
\end{Rem}

A representation of $\tm$ is always a metaplectic tensor product (\cite{Takeda2016}, Lemma 4.5). Moreover, we have the following useful lemmas.

\begin{Lem}[\cite{Takeda2016} Lemma 4.6]\label{lem:takeda 2}
Let $\pi$ and $\pi'$ be irreducible admissible representations of $\tm$. Then $\pi$ and $\pi'$ are equivalent if and only if $\pi|_{Z_{\tgl_r}\tm^{(n)}}$ and $\pi'|_{Z_{\tgl_r}\tm^{(n)}}$ have an equivalent constituent.
\end{Lem}

\begin{Lem}[\cite{Takeda2016} Proposition 4.7] \label{lem:takeda 3}We have
\[
\Ind_{Z_{\tgl_r}\tmn}^{\tm}\pi_\omega^{(n)}=m\pi_\omega
\]
for some finite multiplicity $m$, so every constituent of $\Ind_{Z_{\tgl_r}\tmn}^{\tm}\pi_\omega^{(n)}=m\pi_\omega$ is isomorphic to $\pi_\omega$.
\end{Lem}

Indeed, we can verify that $m=[\widetilde{H}:Z_{\tgl_r}\tmn]$.

\begin{Lem}\label{lem:induction in metaplectic tensor}
We have
\[
\Ind_{\tmn}^{\tm} \pi^{(n)}=m\left(\bigoplus_\xi\pi_\xi\right)
\]
where $m$ is $[\widetilde{H}:Z_{\tgl_r}\tmn]$ and $\xi$ is over the finite set of characters of $Z_{\tgl_r}\tmn$ that are trivial on $\tmn$.
\end{Lem}

\begin{proof}
The proof is the same as in \cite{Takeda2016} Proposition 4.7; see also \cite{Takeda2017} Proposition 3.5.
\end{proof}

\subsection{Examples}\label{sec:examples}

We give some examples of the metaplectic tensor product in this section. The key ingredient in the proof is Lemma \ref{lem:takeda 2}. This allows us to compare irreducible smooth representations of $\tm$ by restricting to $Z_{\tgl_r}\tm^{(n)}$.

Let $\chi$ be a genuine quasicharacter on $Z_{\tgl_r}\tto^n$, and $\omega=\chi|_{Z_{\tgl_r}}$ be the central quasicharacter.  For each $i$, let $\tton_{\ast,i}$ be a maximal abelian subgroup of $\tton_{i}$. Let $\tton_\ast$ be the direct product of $\tton_{\ast,1},\cdots, \tton_{\ast,k}$ with  amalgamated $\mu_n$. Then $\tton_\ast$ is a maximal abelian subgroup of $\tton$. Let $\tto_\ast$ be a maximal abelian subgroup of $\tto$ such that $\tto_\ast\cap \tton=\tton_\ast$.

Let $\chi'$ be an extension of $\chi$ to $\tto_\ast$. We may decompose $\chi'|_{\tton_\ast}$ as
\[
\chi_1\tilde\otimes \cdots \tilde \otimes \chi_k,
\]
where $\chi_i$ is a genuine character on $\tton_{\ast,i}$. Let $\tto_{\ast,i}$ be a maximal abelian subgroup of $\tto_i$ such that $\tton_i\cap \tto_{\ast,i}=\tton_{\ast,i}$. We still use $\chi_i$ to denote an extension of $\chi_i$ to $\tto_{\ast,i}$ (this extension is not unique). When $\chi$ is in  general position, so are $\chi_i$'s. Therefore the principal series representations $I(\chi_i')$ on $\tgl_{r_i}$ are irreducible.

\begin{Thm}
Assume that $\chi$ is in general position. Then the metaplectic tensor product $(I(\chi_1') \tilde\otimes\cdots\tilde\otimes I(\chi_k'))_{\omega}$ is independent on the choices of $\chi_i$. Moreover, as representations of $\tm$,
\[
I(\chi')\cong (I(\chi_1')\tilde\otimes \cdots \tilde\otimes I(\chi_k'))_{\omega}
\]
\end{Thm}

This result shows that, for principal series representations, the metaplectic tensor product can be viewed as an instance of Langlands functoriality on covering groups; see Gan \cite{Gan}.

\begin{proof}
Indeed, the choice of the character $\chi_i$ on $\tto_{\ast,i}$ is up to a character of $\tto_{\ast,i}/\tton_{\ast,i}$. Thus the resulting principal series representations differ by a character that is trivial on $\tgln_{r_i}$. By Proposition \ref{prop:Mezo}, the metaplectic tensor products are still in the same isomorphism class. This proves the well-definedness.

For the second assertion, let us follow the construction of metaplectic tensor product. For $I(\chi_i')|_{\tgln_{r_i}}$, we choose one irreducible constituent $\Ind^{\tgln_{r_i}}_{\tb_{\ast,i}^{\square}}\chi_i'\delta_{B_i}^{1/2}$. Then as representations of $Z_{\tgl_r}\tmn$,
\[
\omega(\Ind^{\tgln_{r_1}}_{\tb_{\ast,1}^{\square}}\chi_1'\delta^{1/2}_{B_1} \tilde\otimes\cdots\tilde\otimes \Ind^{\tgln_{r_k}}_{\tb_{\ast,k}^{\square}}\chi_k'\delta_{B_k}^{1/2})
\cong \omega\Ind_{\tb_\ast^{\square}}^{\tmn} \chi'\delta_M^{1/2}.
\]
This is an irreducible constituent of
\[
(I(\chi_1') \tilde\otimes\cdots\tilde\otimes I(\chi_k'))_{\omega}|_{Z_{\tgl_r}\tmn}.
\]
On the other hand,
\[
\omega\Ind_{\tb_\ast^{\square}}^{\tmn} \chi'\delta_M^{1/2}\cong \Ind^{Z_{\tgl_r}\tmn}_{Z_{\tgl_r}\tb_\ast^{\square}} \chi'\delta_M^{1/2}.
\]
is also an irreducible constituent of $I(\chi')|_{Z_{\tgl_r}\tmn}$. By Lemma \ref{lem:takeda 2}, we are done.

\end{proof}

Next, we turn to exceptional representations.  We start with an exceptional character $\chi$ on $Z_{\tgl_r}\tto^{n}$, and form the exceptional representation $\Theta_{\tm}(\chi')$ as the irreducible quotient of $\Ind_{\tb_\ast}^{\tm}\chi'\delta_M^{1/2}$. The characters $\chi_i'$s are defined as in the previous case.

\begin{Thm}\label{thm:meteplectic tensor of theta}
The metaplectic tensor product $(\Theta(\chi_1') \tilde\otimes\cdots\tilde\otimes \Theta(\chi_k'))_{\omega}$ is well-defined. As representations of $\tm$,
\[
\Theta_{\tm}(\chi')\cong (\Theta(\chi_1') \tilde\otimes\cdots\tilde\otimes \Theta(\chi_k'))_{\omega}.
\]
\end{Thm}

\begin{proof}

Again, we want to show that both sides have an equivalent irreducible constituent when restricted to $Z_{\tgl_r}\tmn$. For the left-hand side, we choose $V_0(\chi'|_{\tton_\ast})$.  This is the unique irreducible subrepresentation of $\Ind_{Z_{\tgl_r}w_{M,0}\tb^{\square}_\ast w_{M,0}^{-1}}^{Z_{\tgl_r}\tmn}{}^{w_{M,0}}\chi'\delta_M^{1/2}$. The Jacquet module of $V_0(\chi'|_{\tton_\ast})$ is
\[
J_{U_M}(V_0(\chi'|_{\tton_\ast}))\cong \ind_{Z_{\tgl_r}w_{M,0} \tton_\ast w_{M,0}^{-1}}^{Z_{\tgl_r}\tton} ({}^{w_{M,0}}\chi'\delta_M^{1/2}).
\]

On the right-hand side, we choose $\omega (V_0(\chi_1')\tilde\otimes \cdots\tilde\otimes V_0(\chi_k')),$ whose Jacquet module is
\[
\begin{aligned}
&\omega(\ind_{w_{\GL_{r_1},0} (\tton_{1,\ast})w_{\GL_{r_1},0}^{-1}}^{\tton_1}({}^{w_{\GL_{r_1},0}}\chi_1 \delta_{B_1}^{1/2}) \tilde\otimes \cdots\tilde\otimes \ind_{w_{\GL_{r_k},0} (\tton_{k,\ast})w_{\GL_{r_k},0,}^{-1}}^{\tton_i}({}^{w_{\GL_{r_k},0}}\chi_k \delta_{B_k}^{1/2}))\\
&\cong\ind_{Z_{\tgl_r}w_{M,0}\tton_\ast w_{M,0}^{-1}}^{Z_{\tgl_r}\tton} ({}^{w_{M,0}}\chi'\delta_M^{1/2}).
\end{aligned}
\]
Thus $\omega (V_0(\chi_1')\tilde\otimes \cdots\tilde\otimes V_0(\chi_k'))$ can be also realized as the unique irreducible subrepresentation of
\[
\Ind_{Z_{\tgl_r}w_{M,0}\tb^{\square}_\ast w_{M,0}^{-1}}^{Z_{\tgl_r}\tmn}{}^{w_{M,0}}\chi'\delta_M^{1/2}.
\]
Therefore, as representations of $Z_{\tgl_r}\tmn$,
\[
V_0(\chi'|_{\tton_\ast})\cong\omega (V_0(\chi_1')\tilde\otimes \cdots\tilde\otimes V_0(\chi_k')).
\]
By Lemma \ref{lem:takeda 2}, we are done.
\end{proof}

\begin{Ex}
Consider the partition $(1^r)$. In this case, $\tm$ is just $\tto$ and the metaplectic tensor product is just the representation theory of $\tto$. The exceptional representation on $\tgl_1$ is  $\Ind_{A}^{\tgl_1}\chi'$, where $A$ is a maximal abelian subgroup of $\tgl_1$, and $\chi'$ is an extension of $\chi:F^{\times n}\to \bc^\times $ to $A$. Notice that $\chi$ is an irreducible constituent of  $\left(\Ind_A^{\tgl_1}\chi'\right)\Big|_{F^{\times n}}$. Let $\chi_1,\cdots, \chi_r$  be characters of $F^{\times n}$. Thus the metaplectic tensor product of $\Ind_A^{\tgl_1}\chi_1', \cdots, \Ind_A^{\tgl_1}\chi_r'$ is $\Ind_{\tto_\ast}^{\tto}(\chi_1\otimes \cdots\otimes \chi_r)'$, where $(\chi_1\otimes \cdots \otimes \chi_k)'$ is an extension of $\chi_1\otimes \cdots \otimes \chi_k$ to $\tto_\ast$.
\end{Ex}

\subsection{Semi-Whittaker functionals}\label{sec:semiwhittaker}

Fix a nontrivial additive character $\psi:F\to \bc^\times$. For a partition $\lambda$ of $r$, let $M=M_\lambda$ be the corresponding Levi subgroup of $\GL_r$. We define a character
\[
\psi_\lambda:U_{M}\to U_{M}/[U_{M},U_{M}]\to \bc^\times
\]
as follows. Let $\alpha$ is a positive simple root in $U_{M}$ and $x_\alpha(a)$  be the one-dimensional unipotent subgroup in $U$ corresponding to the root $\alpha$. We define $\psi_\lambda(x_\alpha(a))=\psi(a)$. We extend this character to $\psi_\lambda:U\to \bc^\times$ via the naive projection $U\to U_M$. Notice this character agrees with the character defined in  Section \ref{sec:notation}. For a smooth representation $(\pi,V)$ of $\tgl_r$, a linear functional $L:V\to\bc$ is called a \textit{$\lambda$-semi-Whittaker functional} if $L(\pi(u)v)=\psi_\lambda(u)L(v)$ for all $u\in U,v\in V$. When $\lambda$ is fixed, we simply call it a  semi-Whittaker functional.

\subsubsection{An explicit formula}

We study semi-Whittaker functionals of exceptional representations. First, we have the following observation for Whittaker functionals of exceptional representations on $\tgl_r$.

Let $\Theta(\chi')$ be an exceptional representation of $\tgl_r$. Recall $\psi_{(r)}:U\to \bc^\times$ is defined as $\psi_{(r)}(u)=\psi(\sum_{i=1}^{r-1}u_{i,i+1})$. Let $d=\dim J_{U,\psi_{(r)}}(\Theta(\chi'))$. If we restrict $\Theta(\chi')$ to $\tgln_r$, we still have $d=\dim J_{U,\psi_{(r)}}(\Theta(\chi')|_{\tgln_r})$. By the exactness of Jacquet functor and Corollary \ref{cor:restriction of theta},
\[
\begin{aligned}
d=&\sum_{x^{-1}\in \tto_\ast\bs \tto/\tton}\dim J_{U,\psi_{(r)}}(V_0({}^x \chi'))\\
=&\sum_{x^{-1}\in \tto_\ast\bs \tto/\tton}\dim J_{U,{}^x\psi_{(r)}}(V_0(\chi')).
\end{aligned}
\]
Therefore,
\[
\sum_{x\in \tton\bs \tto}\dim J_{U,{}^x\psi_{(r)}}(V_0(\chi'))=d[\tton\tto_\ast:\tton]=d[\tto_\ast:\tton_\ast]
\]

Now let us return to the setup of  the metaplectic tensor product. Let
\[
\Theta_{\tm}(\chi')\cong (\Theta(\chi_1')\tilde\otimes \cdots\tilde\otimes \Theta(\chi_k'))_\omega
\]
be an exceptional representation of $\tm$. Let $d_i=\dim J_{U_{\GL_{r_i}},\psi_{(r_i)}}\Theta(\chi_i')$.  We now choose representatives for $\tton_i\bs \tto_i$, and combine them together. This gives  a set of representatives of $\tton\bs \tto$. Thus,
\[
\sum_{x\in \tton\backslash \tto} \dim J_{U_M,{}^x\psi_\lam}(V_0(\chi_1')\otimes \cdots\otimes V_0(\chi_k'))=\prod_{i=1}^k d_i [\tto_{\ast,i}:\tton_{\ast,i}].
\]

\begin{Prop}\label{prop:first formula}
\[
\dim J_{U_M,\psi_\lambda}(\Theta_{\tm}(\chi'))=\frac{\prod_{i=1}^k d_i [\tto_{\ast,i}:\tton_{\ast,i}]}{[\widetilde{H}:\tmn]}.
\]
\end{Prop}

\begin{proof}
Write $\pi^{(n)}=V_0(\chi_1')\otimes \cdots\otimes V_0(\chi_k')$. We have
\[
J_{U_M,\psi_\lambda}(\Ind_{\tmn}^{\tm}\pi^{(n)})\cong\bigoplus_{x\in\tmn\bs\tm} J_{U_M,{}^x\psi_\lambda}(\pi^{(n)})
=\bigoplus_{x\in\tton\bs\tto} J_{U_M,{}^x\psi_\lambda}(\pi^{(n)}).
\]
By Lemma \ref{lem:induction in metaplectic tensor}, the dimension of the left-hand side is
\[
[\widetilde{H}:\tmn]\dim J_{U_M,\psi_\lambda}(\Theta_{\tm}(\chi')).
 \]
 The dimension of the right-hand side is $\prod_{i=1}^k d_i [\tto_{\ast,i}:\tton_{\ast,i}].$ This proves the result.
\end{proof}

Now we proceed to simplify this formula. From now on till the end of this section, we drop the subscript $M$ when the ambient group is $\tm$ to avoid burden on notations. The subscript $i$ indicates the subgroup considered is in the $i$-th block $\tgl_{r_i}$.

Let $\pi$ be an irreducible constituent of the representation $\Theta_{\tm}(\chi')|_{Z_{\tgl_r}\tmn}$.  Recall
\[
\Theta_{\tm}(\chi')|_{Z_{\tgl_r}\tmn}=\bigoplus_{x\in \tto_\ast\bs \tto /Z_{\tgl_r}\tton}{}^x\pi=\bigoplus_{x\in  \tton\tto_\ast\bs \tto}{}^x\pi.
\]
Apply the induction functor $\Ind_{Z_{\tgl_r}\tmn}^{\tm}$ and use Lemma \ref{lem:takeda 3} on the right-hand  side.  This gives
\[
\Ind_{Z_{\tgl_r}\tmn}^{\tm}(\Theta_{\tm}(\chi')|_{Z_{\tgl_r}\tmn} ) =[\tto:\tton\tto_\ast][\widetilde{H}:Z_{\tgl_r}\tmn]\Theta_{\tm}(\chi').
\]
Apply the Jacquet functor $J_{U_M}(-)$. This gives
\[
\Ind_{Z_{\tgl_r}\tmn\cap \tto}^{\tto}J_{U_M}(\Theta_{\tm}(\chi')|_{Z_{\tgl_r}\tmn}) =[\tto:\tton\tto_\ast][\widetilde{H}:Z_{\tgl_r}\tmn]J_{U_M}(\Theta_{\tm}(\chi')).
\]

Comparing the dimensions and using $[\tto:Z_{\tgl_r}\tmn\cap\tto]=[\tm:Z_{\tgl_r}\tmn]$ gives
\[
[\tm:\widetilde{H}]=[\tto:\tton\tto_\ast].
\]
Thus
\[
\begin{aligned}
&[\widetilde{H}:\tmn]=\frac{[\tm:\tmn]}{[\tto:\tton\tto_\ast]} =\frac{[\tto:\tton]}{[\tto:\tton\tto_\ast]}  \\ =&[\tton\tto_\ast:\tton]=[\tto_\ast:\tton_\ast].
\end{aligned}
\]

\begin{Thm}\label{thm:final formula}
\[
\dim J_{U_M,\psi_\lambda}(\Theta_{\tm}(\chi'))=\frac{\prod_{i=1}^k [\tto_{\ast,i}:\tton_{\ast,i}]} {[\tto_\ast:\tton_\ast]}\prod_{i=1}^k d_i.
\]
\end{Thm}

\begin{Rem}
We can see that the same calculation is true for the principal series representations.
\end{Rem}

Let us mention some immediate corollaries.

\begin{Cor}
Suppose $|n|_F=1$. If $r_i>n$, for some $i$, then
\[
J_{U_M,\psi_\lambda}(\Theta_{\tm}(\chi'))=0.
\]
\end{Cor}

\begin{proof}
This is because when $r_i>n$, $d_i=0$.
\end{proof}

\begin{Cor}\label{cor:Vanishing semi whittaker}
Suppose $|n|_F=1$.  Let $\Theta_r(\chi')$ be an exceptional representation of $\tgl_r$. If $r_i>n$ for some $i$, then $J_{U,\psi_\lambda}(\Theta_{r}(\chi'))=0$. In other words, there is no semi-Whittaker functional on $\Theta_{r}(\chi')$.
\end{Cor}

\begin{proof}
In fact,
\[
J_{U,\psi_\lambda}(\Theta_{r}(\chi')) =J_{U_M,\psi_\lambda}(J_R(\Theta_{r}(\chi')))=J_{U_ M,\psi_\lambda}(\Theta_{\tm}({}^{w^{M}}\chi'\cdot\delta_P^{1/2}))=0.
\]
\end{proof}

The following corollaries are true without $|n|_F=1$.

\begin{Cor}
When $r_i\leq n$ for all $i$, $J_{U_M,\psi_\lambda}(\Theta_{\tm}(\chi'))\neq 0.$
\end{Cor}

\begin{Cor}\label{cor:local nonvanishing}
When $r_i\leq n$ for all $i$,  $J_{U,\psi_\lambda}(\Theta_{r}(\chi'))\neq 0.$
\end{Cor}

\subsubsection{Construction of maximal abelian groups}\label{sec:explicit maximal abelian}
We now discuss the construction of maximal abelian subgroups. This helps us simplify the formula further.

Given a maximal isotropic subgroup $\Omega$ of the Hilbert symbol, \cite{KP1984} Section 0.3 provides a way to construct maximal abelian subgroups of $\tto$ under certain assumptions. When $|n|_F=1$, $F^{\times n}\fo^\times$ is a maximal isotropic subgroup of the Hilbert symbol. Let
\[
T_\fo=\{\text{diag}(a_1,\cdots,a_r)\in T:\mathrm{val}(a_i)\equiv 0\mod n\}.
\]
Then $Z_{\tgl_r}\tto_\fo$ is called the standard maximal abelian subgroup of $\tto$, in the sense of \cite{KP1984} Section I.1. We use $\tto_\ast^{\mathrm{st}}$ to denote this subgroup.

\begin{Rem}
Notice that $\tto_\ast^{\mathrm{st}}\cap \tton$ is usually not a maximal abelian subgroup of $\tton$, even for $n=2$. When $n=2,c=0$, a ``canonical'' maximal abelian subgroup was introduced in Bump-Ginzburg \cite{BG1992}. The intersection of their maximal abelian subgroup and $\tton$ is a maximal abelian subgroup of $\tton$.
\end{Rem}

Let $\tton_\fo=\tto_\fo\cap \tgln_r$. The following proposition can be proved by imitating the argument in \cite{KP1984} Section 0.3.
\begin{Prop}\label{prop:maximal abelian 1}
The group $Z_{\tgl_r^{(n)}}\tton_\fo$ is a maximal abelian subgroup of $\tton$.
\end{Prop}

\begin{Rem}\label{rem:abelianexample}
Our calculation in Section \ref{sec:semiwhittaker} relies on the index $[\tto:\tto_\ast]$, which is an invariant of $\tto$. We can compute it by using the standard maximal abelian subgroup $\tto_\ast^{\text{st}}$.
\end{Rem}

\begin{Rem}
When $|n|_F=1$, we give an example of maximal abelian subgroup such that its intersection with $\tton$ is $Z_{\tgl_r^{(n)}}\tton_\fo$. Let
\[
\widetilde{Z}_\ast=\{(zI_r,\zeta)\in \widetilde{Z}: z\in \fo^\times F^{\times\frac{n}{\gcd(n,r(2rc+r-1))}}\}
\]
and
\[
\tto_\fo^{(n')}=\{a\in\tto_\fo:\det(a)\in F^{\times \gcd(n,r)}\}.
\]
Then $\widetilde{Z}_\ast Z_{\tgl_r^{(n)}}\tto_\fo^{(n')}=\widetilde{Z}_\ast \tto_\fo^{(n')}$ is a maximal abelian subgroup of $\tto$ and its intersection with $\tton$ is $Z_{\tgl_r^{(n)}}\tton_\fo$.
\end{Rem}

\begin{Rem}
When $|n|_F\neq 1$, it is usually difficult to construct maximal abelian subgroups of $\tto$. However, when $n\mid r$, the situation is still nice in the following sense. Let $\Omega$ be an isotopic subgroup of the Hilbert symbol. Then by the construction in \cite{KP1984} Section 0.3,
\[
\{
(\mathrm{diag}(t_1,\cdots,t_r),\zeta): t_i\in \Omega, \zeta\in\mu_n
\}
\]
is a maximal abelian subgroup of $\tto$.
\end{Rem}

We now discuss  maximal abelian subgroups of $\tton_M$.  For $1\leq i\leq k$, let $\tton_{\ast,i}$ be a maximal abelian subgroup of $\tton_i$. Let $\tton_\ast$ be the direct product of $\tton_{\ast,1},\cdots, \tton_{\ast,k}$ with amalgamated $\mu_n$. Then $\tton_\ast$ is a maximal abelian subgroup of $\tton$.

Let $\tton_{M,\fo}=\tto_\fo\cap \tmn$.

\begin{Lem}
The group $Z_{\tmn}\tton_{M,\fo}$ is a maximal abelian subgroup of $\tton$.
\end{Lem}

The other maximal abelian subgroup we consider is the standard maximal abelian subgroup $\tto_{M,\ast}^{\mathrm{st}}$.

\subsubsection{An explicit formula, continued}

We now continue the calculation of our explicit formula. Throughout this section, the ambient group is $\tm$. We again use the convention after the proof of Proposition \ref{prop:first formula}. Thus $\tto_{\ast}^{\mathrm{st}}$ ($\tto_{\fo}$, resp.) is $\tto_{M,\ast}^{\mathrm{st}}$ ($\tto_{M,\fo}$, resp.) in the previous section, while $\tto_{\ast,i}^{\mathrm{st}}$ and $\tto_{\fo,i}$ are the corresponding subgroups in the $i$-th block $\tgl_{r_i}.$

\begin{Thm}\label{thm:semi whittaker dimension formula}
When $|n|_F=1$,
\[
\dim J_{U_M,\psi_\lam}(\Theta_{\tm}(\chi'))=\frac{\prod_{i=1}^k [\tto_{\ast,i}^{\mathrm{st}}:\tto_{\fo,i}]} {[\tto_{\ast}^{\mathrm{st}}:\tto_{\fo}]}\prod_{i=1}^k d_i.
\]
\end{Thm}
\begin{proof}
Indeed,
\[
\begin{aligned}
&[ \tto_\ast:\tton_\ast ] =\frac{[\tto:\tton_\ast]}{[\tto:\tto_\ast]} =\frac{[\tto:\tton][\tton:\tton_\ast]}{[\tto:\tto^{\mathrm{st}}_\ast]}\\
=&\frac{[\tto:\tton][\tton:\tton_\ast][\tto^{\mathrm{st}}_{\ast}:\tto_{\fo}]}{ [\tto: \tto_{\fo}]}.
\end{aligned}
\]

Notice that
\[
\dfrac{\prod_{i=1}^k[\tto_i:\tton_i]}{[\tto:\tton]} =\dfrac{\prod_{i=1}^k[\tton_i:\tton_{\ast,i}]}{[\tton:\tton_\ast]} =\dfrac{\prod_{i=1}^k[\tto_i: \tto_{\fo,i}]}{[\tto: \tto_{\fo}]}=1.
\]
Combining with Theorem \ref{thm:final formula}, we obtain the desired formula.
\end{proof}

When $r$ is a multiple of $n$, we have the following uniqueness result. This holds even without the assumption $|n|_F=1$.

\begin{Thm}\label{thm:multiplicity one semi whittaker}
If $r=mn$, and $\lambda=(n^m)$, then $J_{U,\psi_\lambda}(\Theta_{r}(\chi'))$ is one-dimensional.
\end{Thm}

\begin{proof}
When $|n|_F=1$, this follows from Theorem \ref{thm:semi whittaker dimension formula}. Indeed, in this case we have $\gcd(n,2rc+r-1)=1$. Therefore $Z_{\tgl_r}\subset \tto_{\fo}$, and $[\tto_{\ast}^{\mathrm{st}}:\tto_{\fo}]=1$. Similarly, $Z_{\tgl_n}\subset \tto_{\fo,i}$, and $[\tto_{\ast,i}^{\mathrm{st}}:\tto_{\fo,i}]=1$. By Proposition \ref{prop:KP1984}, $d_i=1$ for all $i$. Therefore
$\dim J_{U,\psi_\lambda}(\Theta_{r}(\chi'))=1.$

We now assume $|n|_F\neq 1$. Let $\Omega$ be a maximal isotropic subgroup of the Hilbert symbol.  Then
\[
\tto_{\ast}:=\{(\mathrm{diag}(t_1,\cdots,t_r),\zeta):t_i\in\Omega\}
\]
is a maximal abelian subgroup of $\tto$, and $\tton_\ast:=Z_{\tmn}\cdot (\tto_{\ast}\cap \tton)$ is a maximal abelian subgroup of $\tton$. Notice $Z_{\tmn}=\widetilde{Z_{M^{(n)}}}.$
Moreover, $[\tto_\ast:\tton_\ast]=\dfrac{[\tto:\tto_\ast]}{[\tto:\tton_\ast]}$ and
\[
\frac{\prod_{i=1}^k [\tto_i:\tto_{\ast,i}]}{[\tto:\tto_\ast]}=\frac{\prod_{i=1}^k [\tto_i:\tton_{\ast,i}]}{[\tto:\tton_\ast]}=1.
\]
Combining this uniform description with Theorem \ref{thm:final formula}, we are done.
\end{proof}

\begin{Rem}
Recall that the metaplectic cover $\tgl_r$ depends on an implicit choice of the modulus class $c\in\bz/n\bz$. Our results are true for all $c\in\bz/n\bz$. This is clear for the vanishing result (Corollary \ref{cor:Vanishing semi whittaker}) and nonvanishing result (Corollary \ref{cor:local nonvanishing}). For the uniqueness result, notice that when $r=mn$, $Z_{\tgl_r}=\{z^n I_r:z\in F^\times\}$. This fact is independent of $c$. Thus the proof of Theorem \ref{thm:multiplicity one semi whittaker} is independent of $c$.
\end{Rem}

\begin{Rem}
When $n=2$, this is \cite{BG1992} Proposition 1.3 (i). Indeed, when $r=2k$ and the partition is $(2^k)$, this follows from Theorem \ref{thm:multiplicity one semi whittaker}. When $n=2, r=2k+1$, and $M$ corresponds to the partition $(2^k1)$. In this case, $d_i=1$ for all $i$ and $[\tto_{\ast}^{\mathrm{st}}:\tto_{\fo}]=[F^\times:F^{\times 2}\fo^\times]$. Moreover, $[\tto_{\ast,i}^{\mathrm{st}}:\tto_{\fo,i}]=1$ if $r_i=2$; and $=[F^\times:F^{\times 2}\fo^\times]$ if $r_i=1$. The twisted Jacquet module of $\Theta_{\tm}(\chi')$ is again one-dimensional.
\end{Rem}

\section{Global Theory}\label{sec:semi whittaker global}

\subsection{Theta representations}\label{sec:theta global}

Let $n\geq 2$. Let $F$ be a number field containing a full set of $n$th roots of unity $\mu_n$, and let $\ba$ denote the adeles of $F$. For $r\geq 2$, let $\tgl_r(\ba)$ denote an $n$-fold metaplectic cover of the general linear group, as in Section \ref{sec:global metaplectic cover}.

We recall the definition of the global theta representations. These representations were constructed in \cite{KP1984} using the residues of Eisenstein series as follows. Let $B$ be the standard Borel subgroup of $\GL_r$, and $T\subset B$ denote the maximal torus of $\GL_r$. Let $\underline{s}\in\bc^r$ be a multi-complex variable, and define the character $\mu_{\underline{s}}$ of $T(\ba)$ by $\mu_{\underline{s}}(\mathrm{diag}(a_1,\cdots,a_r))=\prod_i |a_i|^{s_i}$. Let $Z(\tto(\ba))$ denote the center of $\tto(\ba)$. Let $\omega_{\underline{s}}$ be a genuine character of $Z(\tto(\ba))$ such that $\omega_{\underline{s}}=\mu_{\underline{s}}\circ p$ on $\{(t^n,1)|t\in T(\ba)\}$, where $p$ is the canonical projection from $\tto(\ba)$ to $T(\ba)$. Choose a maximal abelian subgroup $A$ of $\tto(\ba)$, extend this character to a character of $A$, and induce it to $\tto(\ba)$. Then extend it trivially to $\tb(\ba)$ using the canonical projection from $\tb(\ba)$ to $\tto(\ba)$, and further induce it to the group $\tgl_r(\ba)$. We abuse the notation slightly and write this induced representation $\Ind_{\tb(\ba)}^{\tgl_r(\ba)}\mu_{\underline{s}} \delta_B^{1/2}$. It follows from \cite{KP1984} that this construction is independent of the choice of $A$ and of the extension of characters. Let $E(\underline{s},g)$ be the Eisensetein series attached to this induced representation. It follows from \cite{KP1984} that when $n(s_i-s_{i+1})=1$ for $1\leq i\leq r-1$, this Eisenstein series has a nonzero residue representation. Let $\Lambda\in\bc^r$ be such a pole, and we write the residue representation as  $\Theta_{r,\Lambda}$. The poles where we take the residues are usually clear in the context, and thus sometimes we omit it from the notation. The global theta representation $\Theta_r$ is the metaplectic restricted tensor product of the local exceptional representations $\Theta_{r,v}$, as explained in Section \ref{sec:global metaplectic cover}. It is shown in \cite{KP1984} Section II that $\Theta_r$ is generic if and only if $r\leq n$.

\subsection{Vanishing results}

\begin{Prop}\label{prop: vanishing local semi whittaker}
Let $\theta$ be in the space of $\Theta_r$. Let $\lambda=(r_1\cdots r_k)$ be a partition of $r$. If there is an $r_i>n$ for some $i$, then
\[
\int\limits_{U(F)\bs U(\ba)}\theta(ug)\psi_\lambda(u) \ du\equiv 0
\]
for all choices of data.
\end{Prop}

\begin{proof}

If
\[
\int\limits_{U(F)\bs U(\ba)}\theta(ug)\psi_\lambda(u) \ du
\]
is nonzero for some choice of data, then the functional $l:\Theta_r\to\bc$ defined by
\[
 \theta\longmapsto
\int\limits_{U(F)\bs U(\ba)}\theta(ug)\psi_\lambda(u) \ du
\]
is nonzero. Recall that $\Theta_r$ is the metaplectic restricted tensor product $\widetilde\otimes_v{}' \Theta_{r,v}$ and the space of $\widetilde\otimes_v{}' \Theta_{r,v}$ is the same as the usual restricted tensor product $\otimes_v{}' \Theta_{r,v}$ (see Section \ref{sec:global metaplectic cover}). We can view  $l$ as a functional on $\otimes_v{}' \Theta_{r,v}$ as well.  We can choose a factorizable vector $\otimes'_v \theta_{0,v}$  such that $l(\otimes'_v\theta_{0,v})\neq 0$.

Let $w$ be a non-Archimedean place of $F$ such that   $|n|_w=1$ and $\Theta_r$ is unramified at $w$. Define a local functional $l_w:\Theta_{r,w}\to \bc$ by
\[
\theta_w \mapsto l(\theta_w \otimes (\otimes'_{v\neq w}\theta_{0,v} )).
\]
By our construction, $l_w$ is nonzero. Now the functional $l_w$ factors through the twisted Jacquet module of $\Theta_{r,w}$ for the character $\psi_{\lam}(u)$ on the group $U(F_w)$. This implies that $J_{U(F_w),\psi_\lam}(\Theta_{r,w})$ is nonzero. This contradicts the local result.

\end{proof}

\subsection{Constant terms I}

Let $\lambda=(r_1 \cdots r_k)$ be a partition of $r$. Let $P_{\lambda}$ be the standard parabolic subgroup of $\GL_r$ with Levi subgroup $M_{\lambda}$ and unipotent radical $U_{\lambda}$.

The goal for this section is to determine the constant term of $\Theta_r$.  We first compute the constant term of the Eisenstein series along $U_\lambda$. This turns out to be a sum of Eisenstein series on $\tm_\lam(\ba)$, over a subset of the Weyl group $W$. We exchange the constant term operator and the multi-residue operator, and the constant terms actually span a ``theta representation'' on $\tm_\lam(\ba)$. We then review the construction of the global metaplectic tensor product in Section \ref{sec:global meta tensor}, and show that the theta representation on $\tm_\lambda(\ba)$ is actually the global metaplectic tensor product of $\Theta_{r_i}$'s.

\begin{Prop}\label{prop:constant term 1}
If $\theta\in\Theta_{r}$, then the constant term
\[
m\mapsto \int\limits_{U_\lambda(F)\bs U_{\lambda}(\ba)}\theta(um) \ du,\qquad m\in \tm_{\lambda}(\ba)
\]
is the residue of an Eisenstein series on $\tm_\lambda(\ba)$.
\end{Prop}

\begin{proof}
Let $\theta(g)=\mathrm{Res}_{\underline{s}=\Lambda}E(\phi,\underline{s},g)$.
We first compute the constant term of the Eisenstein series $E(\phi,\underline{s},g)$ along $P_{\lambda}$.  To do this, we introduce the set $W_{\lambda}$ which consists of elements $w^{-1}$ such that $w^{-1}(\beta)>0$ for any $\Phi_{\lambda}^+$, and $w T w^{-1}\subset M_{\lambda}$. By M{\oe}glin-Waldspurger \cite{MW1995} Proposition 2.1.7(2),
\[
\begin{aligned}
E(\phi,\underline{s},g)_{P_{\lambda}}=&\sum_{w^{-1}\in W_{\lambda}}\sum_{\gamma\in (w Bw^{-1}\cap M_{\lambda})(F)\bs M_{\lambda}(F)} T(w,\underline{s})\phi(\underline{s})(\gamma g)\\
=&\sum_{w^{-1}\in W_{\lambda}} E^{\tm_\lambda}(T(w,\underline{s})\phi(\underline{s}),w\underline{s},g)
\end{aligned}
\]

Let $\Lambda$ denote the pole of $E(\phi, \underline{s}, g)$ as in Section \ref{sec:theta global}. To compute the constant term of theta function along $P_{\lambda}$, we use the fact that the multi-residue operator $\lim_{\underline{s}\to\Lambda}\prod_{i=1}^{r-1}(ns_i-ns_{i+1}-1)$ and the constant term operator commute. Following an argument as in the proof of Offen-Sayag \cite{OS2008} Lemma 2.4, we deduce that after applying the multi-residue operator, the only term left is the one corresponding to $w^{M_\lam}$.

We identify the set of simple roots with $\{(i,i+1):1\leq i\leq r-1\}$. Given $w^{-1}\in W_\lambda$, let $\Delta^1(w)=\{i:\al=(i,i+1), w^{-1}(\al)<0\}$. Notice that by the definition of $W_\lam$, $\Delta^1(w)$ is contained in  $\{r_1, r_1+r_2, \cdots\}$. Then the normalized intertwining operator
\[
N(w,\underline{s})=\prod_{i\in\Delta^1(w)}(ns_i-ns_{i+1}-1)T(w,\underline{s})
\]
is holomorphic at $\Lambda$. Notice that the action of $w$ on $\underline{s}$ is
\[
w(s_1,\cdots,s_r)=(s_{w^{-1}(1)},\cdots, s_{w^{-1}(r)}).
\]
Let
\[
\Delta^2(w)=\{i:w^{-1}(i+1)-w^{-1}(i)=1\}\bs \{r_1,r_1+r_2,\cdots\}.
\]
Then the normalized Eisenstein series
\[
\prod_{i\in\Delta^2(w)}(ns_i-ns_{i+1}-1)E^{\tm_\lambda}(N(w,\underline{s})\phi(\underline{s}),w\underline{s},g)
\]
is holomorphic at $\Lambda$. Thus, the terms corresponding to $w^{-1}$ survives after taking multi-residue if and only if
\[
\Delta^1(w)\cup \Delta^2(w)=\{1,\cdots ,r-1\}.
\]
This implies that $\Delta^1(w)=\{r_1,r_1+r_2,\cdots\}$ and $w$ permutes blocks of $M_\lam$. The only possibility is  $w=w^{M_\lam}$.
Thus we have shown the following identity
\[
\theta(g)_{P_\lambda}=\mathrm{Res}_{\underline{s} =\Lambda}E^{\tm_\lam}(T(w^{M_\lam},\underline{s})\phi(\underline{s}),w^{M_\lam}\underline{s},g).
\]
This finishes the proof.

\end{proof}

If we vary $\theta\in \Theta_r$, then the constant terms of $\theta$'s span an irreducible automorphic representation of $\tm_\lambda(\ba)$. We denote it by $\Theta_{\tm_\lambda}.$ As in the general linear case, $\Theta_{\tm_\lambda}$ is the restricted tensor product of local theta representations of $\tm_\lam(F_v)$.

\subsection{Global metaplectic tensor product}\label{sec:global meta tensor}

The global metaplectic tensor product was first given in \cite{Takeda2016} Section 5, and a simplified version is given in \cite{Takeda2017}. We briefly review the latter construction here.

Assume $(\pi,V_\pi)$ is an automorphic representation of $G$, and $V_\pi$ is a space of functions or maps on the group $G$, and $\pi$ is the representation of $G$ on $V_\pi$ defined by right translation. Let $H\subset G$ be a subgroup. Then we define $\pi||_H$ to be the representation of $H$ realized in the space
\[
V_{\pi||_H}:=\{f|_H:f\in V_{\pi}\}
\]
of restrictions of $f\in V_\pi$ to $H$, on which $H$ acts by right translation.

Let $\pi_i$ be a genuine irreducible automorphic unitary representation of $\tgl_{r_i}(\ba)$. Let $H_i=\GL_{r_i}(F)\tgln_{r_i}(\ba),$ and $\sigma_i=\pi_i||_{H_i}$. Then the restriction $\pi_i|_{H_i}$ is completely reducible (\cite{Takeda2017}, Proposition 3.9). Hence  $\sigma_i$ is a subrepresentation of $\pi_i|_{H_i}$.

Note that $H_i$ is indeed a closed subgroup of $\tgl_{r_i}(\ba)$. By the product formula for the Hilbert symbol and block-compatibility of the cocycle, we have the natural surjection
\[
H_1\times \cdots \times H_k\to M(F)\tmn(\ba).
\]
Then consider the space
\[
V_{\sigma_1}\otimes \cdots \otimes V_{\sigma_k}
\]
as functions on the direct product $H_1\times \cdots \times H_k$, which gives rise to a representation of $H_1\times \cdots \times H_k$. If $\varphi_i\in V_{\sigma_i}$ for $i=1,\cdots,k$, we denote this function by
\[
\varphi_1\otimes \cdots \otimes \varphi_k,
\]
and denote the space generated by those function by $V_\sigma$. These functions can be viewed as ``automorphic forms'' on $M(F)\tmn(\ba)$. The group $M(F)\tmn(\ba)$ acts on $V_{\sigma}$ by right translation.  Denote this representation by $\sigma$. This representation is completely reducible (\cite{Takeda2017} Proposition 3.11).

Fix an irreducible subrepresentation $\tau$ of $\sigma$. Then the abelian group
\[
Z_{\tgl_r(\ba)\cap M(F)\tmn(\ba)}
 \]
 acts as a character $\omega_{\tau}$ (\cite{Takeda2017} Lemma 3.17). Choose a ``Hecke character'' $\omega$ on $Z_{\tgl_r(\ba)}$ by extending $\omega_\tau$. Then one can extend $\tau$ to a representation $\tau_\omega$ on $Z_{\tgl_r(\ba)}M(F)\tmn(\ba)$. Consider the smooth induced representaion
\[
\Pi(\tau_\omega):=\Ind^{\tm(\ba)}_{Z_{\tgl_r(\ba)}M(F)\tmn(\ba)}\tau_\omega.
\]
We can view $\Pi(\tau_\omega)$ as a subrepresentation of $\mathcal{A}(\tm)$, which is the space of automorphic forms on $\tm(\ba)$.
Moreover, $\Pi(\tau_\omega)$ has an irreducible subrepresentation (\cite{Takeda2017} Proposition 3.20). Choose such a representation and denote it by $\pi_\omega$. Then we call $\pi_\omega$  a metaplectic tensor product of $\pi_1, \cdots, \pi_k$ with respect to the character $\omega$ and write
\[
\pi_\omega=(\pi_1\tilde\otimes \cdots \tilde\otimes \pi_k)_\omega.
\]
The representation $\pi_\omega$ has the desired local-global compatibility. Moreover,  it is unique up to equivalence, and depends only on $\pi_1,\cdots,\pi_k$ and $\omega$ (\cite{Takeda2017} Theorem 3.23).

\subsection{Constant term II}

We give the second description of the constant term of the theta function. We  show that the theta representation on $\tm_\lambda(\ba)$ is in fact the global metaplectic tensor product of theta representations on $\tgl_{r_i}(\ba)$.

\begin{Thm}\label{thm:constant term 2}
If $\theta\in\Theta_{r}$, the constant term
\[
m\mapsto \int\limits_{U_\lambda(F)\bs U_{\lambda}(\ba)}\theta(um) \ du,\qquad m\in \tm_{\lambda}(\ba)
\]
is in the space $\Theta_{r_1}\tilde\otimes \cdots \tilde\otimes \Theta_{r_k}$. Indeed,
\[
\Theta_{\tm_\lambda}\cong \Theta_{r_1}\tilde\otimes \cdots \tilde\otimes \Theta_{r_k}.
\]
Here, the global metaplectic tensor product is with respect to the central character $\omega$ of $\Theta_{\tm_\lam}$. The poles that we use to define $\Theta_{r_i}$ are specified in the proof.
\end{Thm}

\begin{proof}
Write $\sigma_i=\Theta_{r_i}||_{H_i}$ for $i=1,\cdots, k$. As explained above, the representation $\sigma_1\tilde\otimes \cdots \tilde\otimes \sigma_k$ descents to a representation $\sigma$  on $M(F)\tmn(\ba)$. It suffices to show that
\[
\Theta_{M_{\lam}}||_{M(F)\tmn(\ba)}\hookrightarrow \sigma.
\]
Notice the space $\sigma$ contains the metaplectic tensor products with respect to all possible characters $\omega.$

Before we prove this claim we would like to introduce some notations. Let $E(\underline{s},g)$ be the Eisenstein series on $\tgl_r(\ba)$ and let $\Lambda$ be the pole to define the theta function. Let $\Lambda_P\in\bc^r$ be the $r$-tuple of complex numbers so that  the corresponding $\mu_{\Lambda_P}$ is the modular quasicharacter $\delta_{P_\lam}$ of $P_\lambda$. Write $\Lambda=(\Lambda_{k},\cdots, \Lambda_1)$, where $\Lambda_i\in\bc^{r_i}$. Write $\Lambda_P=(\Lambda_{P,1},\cdots, \Lambda_{P,k})$ such that $\Lambda_{P,i}\in\bc^{r_i}$. Notice that all the entries in $\Lambda_{P,i}$ are the same.

Let $f\in  \Theta_{M_{\lam}}||_{M(F)\tmn(\ba)}$. This means that $f$ is the restriction of the residue of an Eisenstein series $E^{\tm_\lam}(\underline{s},g)$ to $M(F)\tmn(\ba)$. Indeed, if $g\in M(F)\tmn(\ba)$, then
\[
\begin{aligned}
f(g)=&\mathrm{Res}_{\underline{s}=w^M(\Lambda)+\Lambda_{P}} E(\underline{s},g)\\
=&\mathrm{Res}_{\underline{s}=w^M(\Lambda)+\Lambda_{P}} \sum_{\gamma\in B_M(F)\bs M(F)} \phi(\underline{s})(\gamma g)\\
=&\mathrm{Res}_{\underline{s}=w^M(\Lambda)+\Lambda_{P}} \sum_{\gamma\in B_M^{(n)}(F)\bs M^{(n)}(F)} \phi(\underline{s})(\gamma g).\\
\end{aligned}
\]
The last equality follows from the following fact: $M^{(n)}(F)\hookrightarrow M(F)$ induces a bijection
$B_M^{(n)}(F)\bs M^{(n)}(F)\leftrightarrow B_M(F)\bs M(F).$

By a global analogue of Proposition \ref{prop:decomposition of induced}, we can view $\Ind_{\tb_M^{(n)}(\ba)}^{\tmn(\ba)}\mu_{\underline{s}}\delta_M^{1/2}$ as a subspace of $\Ind_{\tb_M(\ba)}^{\tm(\ba)}\mu_{\underline{s}}\delta_{M}^{1/2}$ (what we need here is actually weaker). Without loss of generality, we may assume that $\phi\in\Ind_{\tb_M^{(n)}(\ba)}^{\tmn(\ba)}\mu_{\underline{s}}\delta_M^{1/2} \subset \Ind_{\tb_M(\ba)}^{\tm(\ba)}\mu_{\underline{s}}\delta_{M}^{1/2}$ and furthermore it is decomposable: $\phi=\phi_1\tilde\otimes \cdots\tilde\otimes \phi_k$, where $\phi_i\in \Ind_{\tb_i^{(n)}(\ba)}^{\tgln_{r_i}(\ba)}\mu_{\underline{s}_i}\delta_{B_i}^{1/2}$. Write $g=\mathrm{diag}(g_1,\cdots,g_k)$ and $ \gamma=\mathrm{diag}(\gamma_1,\cdots,\gamma_k)$. Then we can naturally view $f=f_1\tilde\otimes \cdots\tilde\otimes f_k$, where
\[
f_i(g_i)=\mathrm{Res}_{\underline{s}_i=\Lambda_i+\Lambda_{P,i}} \sum_{\gamma_i\in B_i^{(n)}\bs \GL_{r_i}^{(n)}(F)}\phi_{\underline{s}_i}(\gamma_ig_i).
\]
This means that $f_i\in \Theta_{r_i,\Lambda_i+\Lambda_{P,i}}$. We are done.

\end{proof}

\subsection{Global nonvanishing}

Now we prove the global nonvanishing results. Let $\lambda$ be a partition of $r$. Define the Levi subgroup $M=M_\lam$ as usual. Define the semi-Whittaker functional $\psi_\lambda$ as in the local case. We also write $\psi_\lambda(u)=\psi_1(u_1)\cdots\psi_k(u_k)$ if $u=\mathrm{diag}(u_1,\cdots,u_k)\in U\cap M$.

\begin{Thm}\label{thm:Global Nonvanishing semi whittaker}
If $r_i\leq n$ for all $i$, then
\[
\int\limits_{U(F)\bs U(\ba)} \theta(ug)\psi_\lam(u) \ du
\]
is nonzero for some choices of $\theta\in\Theta_r$ and $g\in \tgl_r(\ba)$.
\end{Thm}

\begin{proof}

Notice that
\[
\int\limits_{U(F)\bs U(\ba)} \theta(ug)\psi_\lambda(u) \ du=\int\limits_{U_M(F)\bs U_M(\ba)} \int\limits_{U_\lambda(F)\bs U_\lam(\ba)}\theta(vug) \ dv \ \psi_\lambda(u) \ du.
\]
By Proposition \ref{prop:constant term 1}, it suffices to show that
\[
\int\limits_{U_M(F)\bs U_M(\ba)} \theta(ug)\psi_\lambda(u) \ du\neq 0
\]
for some choices of $\theta\in \Theta_{\tm}$ and $g\in \tm(\ba)$. We now use notations in the proof of Theorem \ref{thm:constant term 2}. Notice that the character $\omega$ in Theorem \ref{thm:constant term 2} does not contribute anything in this integral. Thus it suffices to show that
\[
\int\limits_{U_M(F)\bs U_M(\ba)}f(u)(1) \psi_\lambda(u) \  du \neq 0
\]
for some $f\in \Ind_{M(F)\tmn(\ba)}^{\tm(\ba)}\sigma$. Here $f(u)$ is in $\sigma$ and we use $f(u)(1)$ to denote its value at $1$. Notice that $U_M(\ba)\subset M(F)\tmn(\ba)$. Thus $f(u)(1)=f(1)(u)$.

Without of loss of generality, we can choose $f$ such that $f(1)$ is a simple tensor $f_1\tilde\otimes \cdots \tilde\otimes f_k$, where $f_i\in \sigma_i$. Moveover, we can choose $f_i$ such that the Whittaker coefficient of $f_i$ is nonzero, i.e. $\int\limits_{U_{\GL_{r_i}}(F)\bs U_{\GL_{r_i}}(\ba)}f_i(u_ig_i)\psi_i(u_i) \ du_i\neq 0$ for all $i$. (This is because when $r_i\leq n$, $\Theta_{r_i}$ is generic.)

Thus,
\[
\begin{aligned}
&\int\limits_{U_M(F)\bs U_M(\ba)}f(u)(1) \psi_\lambda(u) \  du \\
=& \int\limits_{U_M(F)\bs U_M(\ba)}f(1)(u) \psi_\lambda(u) \  du \\
=&\prod_{i=1}^k \int\limits_{U_{\GL_{r_i}}(F)\bs U_{\GL_{r_i}}(\ba)}f_i(u_i)\psi_i(u_i) \  du_i
\neq 0.
\end{aligned}
\]
This proves the theorem.
\end{proof}


\section{Unipotent Orbits and Fourier Coefficients}\label{sec:unipotent orbit}

For the rest of this paper, we turn to the Fourier coefficients associated with general unipotent orbits. In this section, we explain how to associate a set of Fourier coefficients with a unipotent orbit. A general reference for unipotent orbits is Collingwood-McGovern \cite{CM1993}. (The classification of unipotent orbits for classical groups can be found in \cite{CM1993} Chapter 5.) For the local version of this association see \cite{Moeglin1996,MW1987}. For global details see Jiang-Liu \cite{JL2013} and Ginzburg \cite{Ginzburg2006,Ginzburg2014}. The associated Fourier coefficients are described as integration over certain unipotent subgroups, and the metaplectic cocycles do not contribute to any nontrivial factors. To simplify notations, we only describe this association in the non-metaplectic setup.

We work with the global setup. Let $F$ be a number field, and $\ba$ be its adele ring. Fix a nontrivial additive character $\psi:F\bs \ba\to\bc^\times$. The unipotent orbits of $\GL_r$ are parameterized by partitions of $r$. Let $\sco=(p_1\cdots p_k)$ with $p_1+\cdots+p_k=r$ be a unipotent orbit.  We shall always assume $p_1\geq p_2\geq \cdots\geq p_k>0$. To each $p_i$ we associate the diagonal matrix
\[
\mathrm{diag}(t^{p_i-1},t^{p_i-3},\cdots,t^{3-p_i},t^{1-p_i}).
\]
Combining all such diagonal matrices and arranging them  in decreasing order of the powers, we obtain a one-dimensional torus $h_{\sco}(t)$. For example, if $\sco=(3^21)$, then
\[
h_\sco(t)=\mathrm{diag}(t^2,t^2,1,1,1,t^{-2},t^{-2}).
\]

The one-dimensional torus $h_{\sco}(t)$ acts on $U$ by conjugation. Let $\alpha$ be a positive root and $x_\alpha(a)$ be the one-dimensional unipotent subgroup in $U$ corresponding to the root $\alpha$. There is a nonnegative integer $m$ such that
\begin{equation}\label{eq:conjugation}
h_{\sco}(t)x_\alpha(a)h_{\sco}(t)^{-1}=x_\alpha(t^ma).
\end{equation}
On the subgroups $x_\al(a)$ which correspond to negative roots $\alpha$, the torus $h_\sco(t)$ acts with non-positive powers.

Given a nonnegative integer $l$,  we denote by $U_l(\sco)$ the subgroup of $U$ generated by all $x_\alpha(a)$ satisfying the Eq. (\ref{eq:conjugation}) with $m\geq l$. We are mainly interested in $U_l(\sco)$ where $l=2, 3$.

Let
\[
M(\sco)=T\cdot \la x_{\pm\alpha}(a):h_{\sco}(t)x_\alpha(a)h_\sco(t)^{-1}=x_\alpha(a)\ra.
\]
The group $M(\sco)$ acts by conjugation on the abelian group $U_2(\sco)/U_3(\sco)$. If the ground field is algebraically closed, then under this action of $M(\sco)$ on the group $U_2(\sco)/U_3(\sco)$, there is an open orbit. Denote a representative of this orbit by $u_2$. It follows from the general theory that the connected component of the stabilizer of this orbit inside $M(\sco)$ is a reductive group. Denote by $Stab_\sco^0$ this connected component of the stabilizer of $u_2$.

The  $F$-rational points $M(\sco)(F)$ acts on the group of all characters of $U_2(\sco)(F)\bs U_2(\sco)(\ba)$. For each character, its stabilizer is a subgroup of $M(\sco)(F)$ as an algebraic group, and hence it is of the form $C(F)$ for some algebraic group $C$. If the character is such that the connected component of $C$ is isomorphic to $Stab_\sco^0$ over the algebraic closure, it is denoted by $\psi_{U_2(\sco)}.$ Notice that the character $\psi_{U_2(\sco)}$ is not unique.  Given an automorphic function $\varphi(g)$ on $\GL_r(\ba)$ or its cover,  the Fourier coefficients we want to consider are
\[
\int\limits_{U_2(\sco)(F)\bs U_2(\sco)(\ba)}\varphi(ug)\psi_{U_2(\sco)}(u) \ du.
\]
In this way, we associate with each unipotent orbit $\sco$ a set of Fourier coefficients. When the partition is $\sco=(r)$, the Fourier coefficients associated to $\sco$ are the Whittaker coefficients.

In order to perform root exchange as in Sections \ref{sec:root exchange} and \ref{sec:global root exchange} below, we also work with a slightly different torus. Let
\[
h'_{\sco}(t)=\mathrm{diag}(t^{p_1-1}, t^{p_1-3}, \cdots,t^{1-p_1},t^{p_2-1},\cdots,t^{1-p_2}).
\]
Here the first block of size $p_1$ is
\[
\mathrm{diag}(t^{p_1-1}, t^{p_1-3}, \cdots,t^{1-p_1}),
\]
while the remaining part $\mathrm{diag}(t^{p_2-1},\cdots,t^{1-p_2})$ is $h_{(p_2\cdots p_r)}(t)$. For example, if $\sco=(3^21)$, then
\[
h_{\sco}'(t)=(t^2,1,t^{-2}, t^2,1,1,t^{-2}).
\]

The tori $h_{\sco}(t)$ and $h_{\sco}'(t)$ are conjugate by an element in the Weyl group of  $\GL_r$. Let $V_2(\sco)$, $\psi_{V_2(\sco)}$ be the corresponding unipotent subgroup and character, respectively.

Let us recall the partial ordering defined on the set of unipotent orbits. Given $\sco_1=(p_1\cdots p_k)$ and $\sco_2=(q_1\cdots q_l)$, we say that $\sco_1\geq \sco_2$ if $p_1+\cdots+p_i\geq q_1+\cdots+q_i$ for all $1\leq i\leq l$. If $\sco_1$ is not greater than $\sco_2$ and $\sco_2$ is not greater than $\sco_1$, we say that $\sco_1$ and $\sco_2$ are not comparable.

\begin{Def}\label{def:unipotent orbit}
Let $\pi$ be an automorphic representation of $\tgl_r(\ba)$. Let $\sco(\pi)$ denote the set of unipotent orbits of $\GL_r$ defined as follows. A unipotent orbit $\sco\in\sco(\pi)$ if $\pi$ has a nonzero Fourier coefficient which is associated with the unipotent orbit $\sco$, and for all $\sco'>\sco$, $\pi$ has no nonzero Fourier coefficient associated with $\sco'$.
\end{Def}

We already describe this association in the global setup. The corresponding local picture could be described analogously. We omit the details.

\begin{Rem}
It is expected that for any automorphic representation $\pi$, the set $\sco_G(\pi)$ is a singleton (see \cite{Ginzburg2006} Conjecture 5.4). In this paper, the notation $\sco_G(\pi)=\mu$  means that the set $\sco_G(\pi)$ is a singleton, consisting of the orbit $\mu$ only.
\end{Rem}

\section{Unipotent Orbits: Local Results}\label{sec:local unipotent}

We return to the local setup in this section. Fix positive integers $n,r$ such that $|n|_F=1$. Write $r=an+b$, where $a\in\bz_{\geq0}$ and $0\leq b<n$. Let $\Theta=\Theta_r$ be an exceptional representation on $\tgl_r$. The unipotent orbit attached to $\Theta$ is determined in this section. The key ingredients are the results on the semi-Whittaker functionals.  We follow closely the approach given in Jiang-Liu \cite{JL2013}, where they determine the unipotent orbits attached to the residual spectrum of the general linear groups. Here we give a local version with necessary modifications.

\begin{Thm}\label{thm:local unipotent orbit 1}
Let $\sco=(p_1\cdots p_k)$ be a unipotent orbit of $\GL_r$.
\begin{enumerate}[\normalfont (1)]
\item  If $p_1>n$, then $J_{U_{2}(\sco),\psi_{U_2(\sco)}}(\Theta)=0$ (or equivalently, $J_{V_{2}(\sco),\psi_{V_2(\sco)}}(\Theta)=0$).
\item  If $\sco=(n^a b)$, then $J_{U_{2}(\sco),\psi_{U_2(\sco)}}(\Theta)\neq0$ (or equivalently, $J_{V_{2}(\sco),\psi_{V_2(\sco)}}(\Theta)\neq0$).
\end{enumerate}
\end{Thm}

Notice that any unipotent orbit greater than or not comparable with $(n^a b)$ must have $p_1>a$. Thus we obtain the following result.

\begin{Thm}\label{thm:local unipotent orbit 2}
Let $\Theta$ be an exceptional representation of $\tgl_r$. Then
\[
\sco(\Theta)=(n^a b).
\]
\end{Thm}

The rest of this section is devoted to proving Theorem \ref{thm:local unipotent orbit 1}. This theorem is also proved in an unpublished work of Gordan Savin by using the Iwahori-Hecke algebras.

\subsection{A general lemma}\label{sec:root exchange}

We start with a general lemma, which is used repeatedly in this section.

Let $G$ be the rational points of a split algebraic group or a cover of such. Let $\mathfrak{u}$ be a maximal nilpotent Lie subalgebra of $\mathrm{Lie}(G)$. Let $\mathfrak{A,C,X}$ and $\mathfrak{Y}$ be Lie subalgebras of $\mathfrak{u}$, and  let $A,C,X,Y$ be the corresponding unipotent subgroups of $G$. Let $\psi_C$ be a nontrivial character of $C$. We make the following assumptions:

\begin{enumerate}[(a)]
\item $C,X,Y\subset A$.
\item $X$ and $Y$ are abelian, normalize $C$ and preserve $\psi_C$.
\item The commutators $x^{-1}y^{-1}xy$ lie in $C$, for all $x\in X,y\in Y$. In particular, $Y$ normalizes $D=CX$ and $X$ normalizes $B=CY$.
\item $A=D\rtimes Y=B\rtimes X$.
\item The set
\[
\{x\mapsto\psi_C(x^{-1}y^{-1}xy)|y\in Y\}
\]
is the group of all characters of $X$. Moreover, writing $x=\exp E,y=\exp S$, for $E\in\mathfrak{X},S\in\mathfrak{Y}$, we have
\[
\psi_C(xyx^{-1}y^{-1})=\psi((E,S))
\]
where $(\ ,\ )$ is a nondegenerate, bilinear pairing between $\mathfrak{X}$ and $\mathfrak{Y}$.
\end{enumerate}

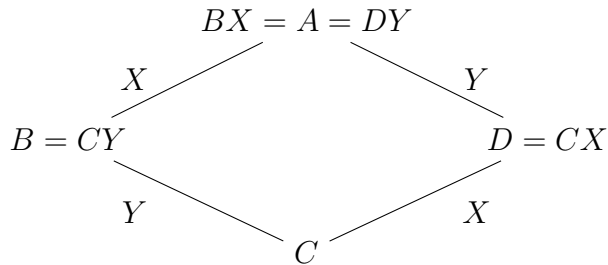
\[
\begin{tikzpicture}
\node(a) {$BX=A=DY$};
\node(b) [below left=1cm and 0.7cm of a] {$B=CY$};
\node(d) [below right=1cm and 0.7cm of a] {$D=CX$};
\node(c) [below =2.5cm of a] {$C$};
\draw[-] (a) to (b);
\draw[-] (a) to (d);
\draw[-] (b) to (c);
\draw[-] (d) to (c);
\node(x1) [below left=0.2cm and 0.4cm of a] {$X$};
\node(y1) [below right=0.2cm and 0.4cm of a] {$Y$};
\node(y2) [below = 1.2cm of x1] {$Y$};
\node(x2) [below = 1.2cm of y1] {$X$};
\end{tikzpicture}
\]

\begin{Lem}\label{lem:local root exchange}
Assume (a)-(e). Let $\pi$ be a smooth representation of $A$. Extend $\psi_C$ trivially to characters $\psi_B$ of $B$ and $\psi_D$ of $D$. Then we have an isomorphism of $C$-modules
\[
J_{B,\psi_B}(\pi)\cong J_{D,\psi_D}(\pi).
\]
Moreover,
\[
J_{C,\psi_C}(\pi)=0\Longleftrightarrow J_{D,\psi_D}(\pi)=0 \Longleftrightarrow J_{B,\psi_B}(\pi)=0.
\]
\end{Lem}

\begin{proof}
The first isomorphism is proved in Ginzburg-Rallis-Soudry \cite{GRS1999} Section 2.2. We now prove the second statement. By symmetry, it suffices to prove
\[
J_{C,\psi_C}(\pi)=0\Longleftrightarrow J_{D,\psi_D}(\pi)=0.
\]
Clearly if $J_{C,\psi_C}(\pi)=0$, then
\[
J_{D,\psi_D}(\pi)=J_X(J_{C,\psi_C}(\pi))=0.
\]
Conversely, suppose $J_{D,\psi_D}(\pi)=0$. There is a natural map
\[
T:J_{C,\psi_C}(\pi)\to J_{D,\psi_D}(\pi)=0
\]
over $D$. This induces a map of $A$-modules
\[
i:J_{C,\psi_C}(\pi)\to \Ind_D^A(J_{D,\psi_D}(\pi))=0.
\]
It is shown in \cite{GRS1999} Section 2.2 that $i$ is injective. Thus $J_{C,\psi_C}(\pi)=0$.
\end{proof}

When $X$ and $Y$ are root subgroups, the above lemma is the local version of the root exchange in Friedberg-Ginzburg \cite{FG} Section 2.2 and Ginzburg \cite{Ginzburg2015} Section 2.2.2. This is always the case in our application. The above assumptions can always be verified by the Steinberg relations.

\subsection{Root exchange}
Given a unipotent orbit $\sco=(p_1\cdots p_k)$, we define several unipotent subgroups of $U$. Let  $U_{\sco}$ be the subgroup of $U$ consisting elements of the form
\[
u=\begin{pmatrix}
u_1 & n_1\\
&u_2
\end{pmatrix},
\]
where $u_1$ is a unipotent matrix in $\GL_{p_1}$, $n_1\in \mathrm{Mat}_{p_1\times (n-p_1)}$ with the last row being zero, and $u_2\in U_2((p_2\cdots p_k))\subset \GL_{r-p_1}$. We define a character $\psi_{U_\sco}: U_{\sco}\to \bc^\times$ as the product of the Whittaker character on $u_1$ and $\psi_{U_2((p_2\cdots p_k))}$ on $u_2$. We also define a unipotent subgroup $U_{\sco}'$ of $U_{\sco}$ by removing all the root subgroups $U_\al$ in the $n_1$ part for all $\al$ such that
\begin{equation}\label{eq:definition of removing root}
h'_{\sco}(t)x_\alpha(a)h'_{\sco}(t)^{-1}=x_\alpha(ta).
\end{equation}
Let  $\psi_{U'_\sco}$ be the restriction of $\psi_{U_{\sco}}$ to $U_\sco'$.

\begin{Rem}\label{rem:same parity}
If $p_i$'s have the same parity, then $U_{\sco}=U_{\sco}'$.
\end{Rem}

\begin{Lem}\label{lem:consequence of root exchange}
Let $\pi$ be a smooth representation of $\tgl_r$.
\begin{enumerate}[\normalfont(1)]
\item
\[
J_{V_2(\sco),\psi_{V_2(\sco)}}(\pi)\cong J_{U_{\sco}',\psi_{U'_\sco}}(\pi).
\]
\item
\[
J_{V_2(\sco),\psi_{V_2(\sco)}}(\pi)=0
\]
if and only if
\[
J_{U_{\sco},\psi_{U_\sco}}(\pi)=0.
\]
\item If $p_i$'s have the same parity, then
\[
J_{V_2(\sco),\psi_{V_2(\sco)}}(\pi)\cong J_{U_{\sco},\psi_{U_\sco}}(\pi).
\]
\end{enumerate}
\end{Lem}

\begin{proof}
Part (3) is clear from part (1) and Remark \ref{rem:same parity}. We first prove part (1).
The strategy is to use the root exchange lemma. Notice that any element of $V_2(\sco)$ has the following form:
\[
u=\begin{pmatrix}
u_1 & q\\
0 & u_2
\end{pmatrix}
\begin{pmatrix}
I_{p_1} & 0\\
p & I_{n-p_1}
\end{pmatrix},
\]
where $u_1\in\GL_{p_1}$ and $u_2\in U_{(p_2\cdots p_k)} \subset \GL_{n-p_1}$ are unipotent matrices, and $p\in \mathrm{Mat}_{p_1\times(n-p_1)}$ and $q\in \mathrm{Mat}_{(n-p_1)\times p_1}$ are certain matrices to be described later. The character $\psi_{V_2(\sco)}$ is the product of Whittaker character on $u_1$ and $\psi_{U_2((p_2\cdots p_k))}$. We use the simple roots in $u_1$ to move root subgroups contained in $p$ to $q$. The desired twisted Jacquet module is obtained after we finish this process.

Let us give more details in the case $\sco=(p_1p_2)$.  The general case follows by the same argument. There are two cases to consider, depending on the parity of $p_1-p_2$.

Case 1: $p_1-p_2$ is even. Notice that in this case part (1) implies part (2) immediately.

We can write $u\in V_2(\sco)$ as
\[
u=\begin{pmatrix}
u_1 & n_1\\
n_2& u_2
\end{pmatrix}.
\]
Here $u_1\in \GL_{p_1},u_2\in \GL_{p_2}$ are unipotent matrices, and
\[
n_1=\begin{pmatrix}a_1\\b_1\\c_1\end{pmatrix}\in\mathrm{Mat}_{p_1\times p_2},
\]
where
\[
a_1\in\mathrm{Mat}_{(\frac{p_1-p_2}{2})\times p_2}, b_1\in \mathrm{Mat}_{p_2\times p_2} \text{ is nilpotent}, c_1=0\in \mathrm{Mat}_{(\frac{p_1-p_2}{2})\times p_2},
\]
and
\[
n_2=\begin{pmatrix}0 & b_2 & c_2\end{pmatrix}\in\mathrm{Mat}_{p_2\times p_1}
\]
where
\[
0\in\mathrm{Mat}_{p_2\times (\frac{p_1-p_2}{2}+1)}, b_2\in \mathrm{Mat}_{p_2\times p_2}\text{ is upper triangular}, c_2\in \mathrm{Mat}_{p_2\times(\frac{p_1-p_2}{2}-1)}.
\]

Now we apply the root exchange lemma. For the first column of $b_2$, the only nonzero entry is the root subgroup corresponding to the (negative) root $(p_1+1,\frac{p_1-p_2}{2}+2)$. We now describe the groups $A,B, C, D,X,Y$ in Section \ref{sec:root exchange} in our current setting.   Let $A=V_2(\sco)$. Let $C$, $X$ and $Y$ be the root subgroups corresponding to the roots $(\frac{p_1-p_2}{2}+1,\frac{p_1-p_2}{2}+2)$, $(p_1+1,\frac{p_1-p_2}{2}+2)$ and $(\frac{p_1-p_2}{2}+1,p_1+1)$, respectively. This determines the groups $B$ and $D$. Notice that the character $\psi_{V_2(\sco)}$ is nontrivial on $C$. After applying Lemma \ref{lem:local root exchange}, we replace the root $(p_1+1,\frac{p_1-p_2}{2}+2)$ by the (positive) root $(\frac{p_1-p_2}{2}+1,p_1+1)$ in the twisted Jacquet module.

Similarly, the $i$th column of $b_2$ has $i$ nonzero entries corresponding to the roots
\[
\left(j,\frac{p_1-p_2}{2}+i+1\right),\qquad j=p_1+1,\cdots,p_1+i.
\]
We let $X$ be the group generated by the root subgroups corresponding to these roots. Let $C$ be the root subgroup corresponding to the root $(\frac{p_1-p_2}{2}+i,\frac{p_1-p_2}{2}+i+1)$. Let $Y$ be the group generated by the root subgroups corresponding to the roots
\[
\left(\frac{p_1-p_2}{2}+i,j\right),\qquad j=p_1+1,\cdots,p_1+i.
\]
These are exactly the missing entries in the $i$th row in $b_1$. By applying Lemma \ref{lem:local root exchange}, we replace $X$ by $Y$ in the twisted Jacquet module.  The $c_2$ part can be handled similarly. Indeed, using the simple roots in $u_1$, the entries in $c_2$ are moved to the first $(\frac{p_1-p_2}{2}-1)$ rows of $c_1$. Thus, in this case, we have shown that
\[
J_{V_2(\sco),\psi_{V_2(\sco)}}(\pi)\cong J_{U'_\sco,\psi_{U'_\sco}}(\pi).
\]

Case 2: $p_1-p_2$ is odd. The proof of part (1) is the same as Case 1, with minor differences. Indeed, $u\in V_2(\sco)$ can be written as
\[
u=\begin{pmatrix}
u_1 & n_1\\
n_2& u_2
\end{pmatrix}.
\]
Here $u_1\in \GL_{p_1},u_2\in \GL_{p_2}$ are unipotent matrices, and
\[
n_1=\begin{pmatrix}a_1\\b_1\\c_1\end{pmatrix}\in\mathrm{Mat}_{p_1\times p_2},
\]
where
\[
a_1\in\mathrm{Mat}_{(\frac{p_1-p_2-1}{2})\times p_2}, b_1\in \mathrm{Mat}_{p_2\times p_2} \text{ is nilpotent}, c_1=0\in \mathrm{Mat}_{(\frac{p_1-p_2+1}{2})\times p_2},
\]
and
\[
u_2=\begin{pmatrix}0 & b_2 & c_2\end{pmatrix}\in\mathrm{Mat}_{p_2\times p_1}
\]
where
\[
0\in\mathrm{Mat}_{p_2\times (\frac{p_1-p_2+1}{2})}, b_2\in \mathrm{Mat}_{p_2\times p_2}\text{ is nilpotent}, c_2\in \mathrm{Mat}_{p_2\times(\frac{p_1-p_2-1}{2})}.
\]

There is no element in the first column of $b_2$, and the first entry of $b_1$ is missing. For the second column of $b_2$, the only nontrivial entry corresponds to the root $(p_1+1,\frac{p_1-p_2+1}{2}+2)$. Let $X$ be the root subgroup corresponding to this root. We now let $C$ and $Y$ be the root subgroups corresponding to the roots $(\frac{p_1-p_2+1}{2}+1,\frac{p_1-p_2+1}{2}+2)$ and $(\frac{p_1-p_2+1}{2}+1,p_1+1)$. By applying Lemma \ref{lem:local root exchange}, we can replace $X$ by $Y$ in the twisted Jacquet module.  The group $Y$ gives the first entry of the second row of $b_1$. Now we only miss the second entry in the second row of $b_1$.

Similarly, the $(i+1)$th column of $b_2$ has $i$ entries corresponding to the roots
\[
(j,\frac{p_1-p_2+1}{2}+i+1),\qquad j=p_1+1,\cdots,p_1+i.
\]
Let $X$ be the group generated by the root subgroups corresponding to these roots. Let $C$ be the root subgroup corresponding to the root $(\frac{p_1-p_2}{2}+i,\frac{p_1-p_2+1}{2}+i+1)$. Let $Y$ be the group generated by the root subgroups corresponding to the roots
\[
(\frac{p_1-p_2+1}{2}+i,j),\qquad j=p_1+1,\cdots,p_1+i.
\]
By Lemma \ref{lem:local root exchange}, $X$ can be replaced by $Y$ in the twisted Jacquet module.
Thus, after this process, we only miss the $(i+1)$th entry in the $(i+1)$th row in $b_1$. The $c_2$ part can be handled similarly, and the entries in $c_2$ are moved to the first $(\frac{p_1-p_2+1}{2}-1)$ rows of $c_1$. The missing entries in $b_1$  are the diagonal entries, which are exactly the root subgroups that are removed in the definition of $U'_\sco$; see Eq. (\ref{eq:definition of removing root}). This finishes the proof of part (1).

For part (2), let $Y$ be the subgroup of $V_2(\sco)$ such that $u_1=I,u_2=I$, $n_2,a_1,c_1=0$, and $b_1$ is diagonal. Then we can verify that $Y$ normalizes $U_{\sco}'$ and preserves $\psi_{U_\sco}$. Moreover, $U_\sco' Y=U_\sco$. By Lemma \ref{lem:local root exchange},
\[
J_{U_{\sco},\psi_{\sco}}(\pi)=0
\]
if and only if
\[
J_{U_{\sco}',\psi_{\sco}'}(\pi)=0
\]
if and only if
\[
J_{V_2(\sco),\psi_{V_2(\sco)}}(\pi)=0.
\]

For the general case, we need to proceed inductively. Notice that if we perform root exchange on the $p_3$ part using $p_1$, what is done in the previous steps is unchanged. Therefore, the lemma is true for a general unipotent orbit $\sco$.
\end{proof}

\subsection{Vanishing results}

Now we prove the vanishing property of the twisted Jacquet modules of $\Theta$ attached to the unipotent orbits either greater than or not comparable with  $(n^ab)$.

Let $V_{1^{m-1},r-m+1}$ be the unipotent radical of the parabolic subgroup $P_{1^{m-1},r-m+1}$ with Levi part $\GL_1^{\times (m-1)}\times \GL_{r-m+1}$. Let
\[
\psi_{m-1}(v)=\psi(v_{1,2}+\cdots+v_{m-1,m}),
\]
and
\[
\tilde\psi_{m-1}(v)=\psi(v_{1,2}+\cdots+v_{m-2,m-1})
\]
be two characters of $V_{1^{m-1},r-m+1}$. Notice that $(V_{1^{m-1},r-m+1}, \psi_{m-1})$ is the same as $(U_{\sco}, \psi_{U_{\sco}})$ where $\sco=(m1^{r-m})$.

We consider  slightly more general characters. Let $m'\geq m$, and
\[
\underline{\epsilon}=(\epsilon_{m},\epsilon_{m+1},\cdots, \epsilon_{m'-1})\in F^{m'-m}.
\]
Let
\[
\psi_{m-1,\underline{\epsilon}}(v)=\psi(v_{1,2}+\cdots+v_{m-1,m}+\epsilon_{m}v_{m,m+1}+\cdots +\epsilon_{m'-1}v_{m'-1,m'})
\]
be a character of $V_{1^{m'-1},r-m'+1}$.

\begin{Lem}\label{lem:vanishing unipotent orbit 1}
If $m>n$, then
\[
J_{V_{1^{m'-1},r-m'+1},\psi_{m-1,\underline{\epsilon}}}(\Theta)=0.
\]
In particular,
\[
J_{V_{1^{m-1,r-m+1},\psi_{m-1}}}(\Theta)=0.
\]
\end{Lem}

\begin{proof}
We prove this by induction on $r-m'$. When $r=m'\geq m$, the pair $(V_{1^{r-1},1},\psi_{r-1,\underline{\epsilon}})$ can only be $(U,\psi_{\lam})$ where $\lambda$ is a partition of the form $(m''\cdots)$ with some $m''\geq m$.  The result follows Corollary \ref{cor:Vanishing semi whittaker} since $m''\geq m>n$.

Now assume the result is true for $m'$ and we prove it for $m'-1$ if $m'-1\geq m$. Define $R_{m'-1}$ to be the subgroup of $U$ such that any element $u=(u_{j,l})\in R_{m'-1}$, $u_{j,l}=0$, unless $j=m'-1$. The group $R_{m'-1}$ acts on $V_{1^{m'-2},r-m'+2}$. For any character $\xi$ of $R_{m'-1}$,
\[
J_{R_{m'-1},\xi}(J_{V_{1^{m'-2},r-m'+2},\psi_{m'-2,\underline{\epsilon}}}(\Theta))=0
\]
by induction. This implies
\[
J_{V_{1^{m'-2},r-m'+2},\psi_{m'-2},\underline{\epsilon}}(\Theta)=0.
\]
This finishes the proof.
\end{proof}

\begin{Lem}\label{lem:vanishing unipotent orbit 2}
\[
J_{V_{1^{n-1},r-n+1},\psi_{n-1}}(\Theta)\cong J_{V_{1^{n},r-n},\tilde{\psi}_{n}}(\Theta).
\]
\end{Lem}

\begin{proof}
The group $R_{n}$ acts on $V_{1^{n-1},r-n+1}$. For any nontrivial character $\xi$ of $R_{n}$,
\[
J_{R_{n},\xi}(J_{V_{1^{n-1},r-n+1},\psi_{n-1}}(\Theta))=0
\]
by Lemma \ref{lem:vanishing unipotent orbit 1}. Therefore, the action of $R_{n}$ on $J_{V_{1^{n-1},r-n+1},\psi_{n-1}}(\Theta)$ is trivial, and
\[
J_{V_{1^{n-1},r-n+1},\psi_{n-1}}(\Theta)\cong J_{V_{1^{n},r-n},\tilde\psi_{n}}(\Theta).
\]
\end{proof}

Now we are ready to prove Theorem \ref{thm:local unipotent orbit 1} part (1). Indeed,  since $p_1>n$,
\[
J_{U_{\sco},\psi_{\sco}}(\Theta)=J_{\ast}(J_{V_{1^{p_1-1},r-p_1+1},\psi_{p_1-1}}(\Theta))=0.
\]
Here $\ast$ is some unipotent subgroup of $V_2(\sco)$.
By Lemma \ref{lem:consequence of root exchange}, this implies
\[
J_{V_2(\sco),\psi_{V_2(\sco)}}(\Theta)=0.
\]

\subsection{Nonvanishing results}

In this subsection, $\sco=(n^ab)$. For $1\leq i\leq a$, consider $V_{1^{in-1},r-in+1}$ and its characters attached to the partitions $(n^i)$ and $(n^{i-1}(n+1))$, respectively. We can prove the following lemma by using the same arguments as in Lemma \ref{lem:vanishing unipotent orbit 1} and \ref{lem:vanishing unipotent orbit 2}.

\begin{Lem}\label{lem:vanishing unipotent orbit 3}
\hspace{2em}
\begin{enumerate}[\normalfont (1)]
\item $J_{V_{1^{in},r-in},\psi_{(n^{i-1}(n+1))}}(\Theta)=0$.
\item $J_{V_{1^{in-1},r-in+1},\psi_{(n^{i})}}(\Theta) \cong J_{V_{1^{in},r-in},\psi_{(n^{i})}}(\Theta).$
\end{enumerate}
\end{Lem}

Now we prove the following nonvanishing result (Theorem \ref{thm:local unipotent orbit 1} part (2)).

\begin{Prop}\label{prop:nonvanishing 1}
$J_{V_2(\sco),\psi_{V_2(\sco)}}(\Theta)\neq 0.$
\end{Prop}

\begin{proof}
It suffices to show that $J_{U_\sco,\psi_{U_\sco}}(\Theta)\neq 0$. Indeed,
\[
\begin{aligned}
J_{U_{\sco},\psi_{U_\sco}}(\Theta)=&J_{U_2((n^{a-1}b)),\psi_{U_2((n^{a-1}b))}}(J_{V_{1^{n-1},r-n+1},\psi_{(n)}}(\Theta)) \\ \cong &J_{U_2((n^{a-1}b)),\psi_{U_2((n^{a-1}b))}}(J_{V_{1^{n},r-n},\psi_{(n)}}(\Theta)).
\end{aligned}
\]
Here, $U_2((n^{a-1}b))$ is viewed as a subgroup of $U$ via the  embedding $u\mapsto \mathrm{diag}(I_n, u)$. Now we  apply root exchange in $U_2((n^{a-1}b))$. The root exchange does not change anything we did in the previous step. Thus,
\[
J_{U_2((n^{a-1}b)),\psi_{U_2((n^{a-1}b))}}(J_{V_{1^{n},r-n},\psi_{(n)}}(\Theta))\neq 0
\]
if and only if
\[
J_{U_{(n^{a-1}b)},\psi_{U_{(n^{a-1}b)}}}(J_{V_{1^{n},r-n},\psi_{(n)}}(\Theta))\neq 0.
\]
Here, $U_{(n^{a-1}b)}$ is again viewed as a subgroup of $U$ via the same embedding. By Lemma \ref{lem:vanishing unipotent orbit 3},
\[
\begin{aligned}
&J_{U_{(n^{a-1}b)},\psi_{U_{(n^{a-1}b)}}}(J_{V_{1^{n},r-n},\psi_{(n)}}(\Theta))\\
=&J_{U_2((n^{a-2}b)),\psi_{U_2((n^{a-2}b))}}(J_{V_{1^{2n-1},r-2n+1},\psi_{(n^2)}}(\Theta))\\
\cong &J_{U_2((n^{a-2}b)),\psi_{U_2((n^{a-2}b))}}(J_{V_{1^{2n},r-2n},\psi_{(n^2)}}(\Theta)).
\end{aligned}
\]
Here, $U_2((n^{a-2}b))$ is viewed as a subgroup of $U$ via $u\mapsto \mathrm{diag}(I_{2n},u)$.

Now we repeat this process inductively. This implies that  $J_{V_2(\sco),\psi_{V_2(\sco)}}(\Theta)\neq 0$ if and only if
\[
0\neq J_{U_2((b)),\psi_{U_2((b))}}(J_{V_{1^{an},r-an},\psi_{(n^a)}}(\Theta))=J_{U,\psi_{((n^ab))}}(\Theta).
\]
The result follows from the nonvanishing results of the semi-Whittaker functionals (Corollary \ref{cor:local nonvanishing}).
\end{proof}

Now suppose that $n, b$ have the same parity. By Lemma \ref{lem:consequence of root exchange} part (3),  in all the steps of the  above proof, we actually obtain isomorphisms of twisted  Jacquet modules, instead of ``if and only if'' statements. This proves the following result.
\begin{Prop}
When $n$ and $b$ have the same parity,
\[
J_{V_2(\sco),\psi_{V_2(\sco)}}(\Theta)\cong J_{U,\psi_\lambda}(\Theta),
\]
where $\lambda$ is the partition $(n^ab)$.
\end{Prop}

When $r$ is a multiple of $n$, combining with Theorem \ref{thm:multiplicity one semi whittaker}, we obtain the following uniqueness result.
\begin{Thm}\label{thm:multiplicity one unipotent orbit}
When $r=mn$ and $\sco=(n^m)$,
\[
\dim J_{U_2(\sco),\psi_{U_2(\sco)}}(\Theta)= \dim J_{V_2(\sco),\psi_{V_2(\sco)}}(\Theta)=1.
\]
\end{Thm}

\begin{Rem}

We already proved the local results at good primes. At bad primes, these statements would be valid once we have the corresponding vanishing results of semi-Whittaker functionals.

\end{Rem}

This new unique functional in Theorem \ref{thm:multiplicity one unipotent orbit} is valuable and it already finds applications in Rankin-Selberg integrals for covering groups. The first -- doubling constructions for covering groups -- will be discussed in Section \ref{sec:WSS}. This unique functional also plays a
role in a new-way integral (Euler products with non-unique models) for covering groups; see
Ginzburg \cite{Ginzburg2016}.

\section{Unipotent Orbits: Global Results}\label{sec:global unipotent}
We are back to the global situation in this section. Let $\Theta_r$ be the global theta representation on $\tgl_r(\ba)$, defined in Section \ref{sec:theta global}. Let $(n^ab)$ be the unipotent orbit of $\GL_r$ as in  Section \ref{sec:local unipotent}. We determine $\sco(\Theta_r)$ in this section.

\subsection{Root exchange lemma: global version}\label{sec:global root exchange}

The following global root exchange lemma is proved in \cite{JL2013} Lemma 5.2; see also Ginzburg-Rallis-Soudry \cite{GRS2011} Section 7.1. This is the global version of Lemma \ref{lem:local root exchange}.

Let $C$ be an $F$-subgroup of a maximal unipotent subgroup of $\GL_r$, and let $\psi_C$ be a nontrivial character of $C(F)\bs C(\ba)$.  Let $X, Y$ be two unipotent $F$-subgroups satisfying the following conditions:

\begin{enumerate}[\normalfont (a)]
\item $X$ and $Y$ are abelian and normalize $C$;
\item $X(\ba)$ and $ Y(\ba)$ preserve $\psi_C$;
\item $X\cap C$ and $ Y\cap C$ are normal in $X$ and $ Y$, respectively;
\item $\psi_C$ is trivial on $(X\cap C)(\ba)$ and $(Y\cap C)(\ba)$;
\item $[ X,Y]\subset C$;
\item there is a nondegenerate pairing $( X\cap C)(\ba)\times ( Y\cap C)(\ba)\to \bc^\times$, given by $(x,y)\mapsto \psi_C([x,y])$, which is multiplicative in each coordinate, and identifies $( Y\cap C)(F)\bs  Y(F)$ with the dual of $X(F)(X\cap C)(\ba)\bs X(\ba)$, and $( X\cap C)(F)\bs X(F)$ with the dual of $Y(F)(Y\cap C)(\ba)\bs  Y(\ba)$.
\end{enumerate}

Let $B=CY$ and $D=C X$, and extend $\psi_C$ trivially to characters of $B(F)\bs B(\ba)$ and $D(F)\bs D(\ba)$, which are denoted by $\psi_B$ and $\psi_D$, respectively.

\begin{Lem}
Assume the quadruple $(C,\psi_C,X,Y)$ satisfies the above conditions. Let $f$ be an automorphic form on $\tgl_r(\ba)$. Then
\[
\int\limits_{C(F)\bs C(\ba)}f(cg)\psi_C^{-1}(c) \ dc\equiv 0,\qquad \forall g\in \tgl_r(\ba),
\]
if and only if
\[
\int\limits_{B(F)\bs B(\ba)}f(ug)\psi_B^{-1}(u) \ du\equiv 0,\qquad \forall g\in \tgl_r(\ba),
\]
if and only if
\[
\int\limits_{D(F)\bs D(\ba)}f(ug)\psi_D^{-1}(u) \ du\equiv 0,\qquad \forall g\in \tgl_r(\ba).
\]
\end{Lem}

\subsection{Vanishing results}

\begin{Prop}\label{prop:vanishing global unipotent orbit}
Let $\theta$ be in the space of $\Theta_r$. Let $\sco$ be a unipotent orbit which is greater than or not comparable to $(n^ab)$. Then the integral
\[
\int\limits_{U_2(\sco)(F)\bs U_2(\sco)(\ba)}\theta(ug)\psi_{U_2(\sco)}(u) \ du
\]
is zero for all choices of data.
\end{Prop}

\begin{proof}

As in the case of semi-Whittaker coefficients, the global vanishing result follows from the local vanishing result.  If
\[
\int\limits_{U_2(\sco)(F)\bs U_2(\sco)(\ba)}\theta(ug)\psi_{U_2(\sco)}(u) \ du
\]
is nonzero for some choice of data, then the functional $l:\Theta_r\to\bc$ defined by
\[
 \theta\longmapsto
\int\limits_{U_2(\sco)(F)\bs U_2(\sco)(\ba)}\theta(ug)\psi_{U_2(\sco)}(u) \ du
\]
is nonzero. As explained in the proof of Proposition \ref{prop: vanishing local semi whittaker}, we can choose a factorizable vector $\otimes'_v \theta_{0,v}$  such that $l(\otimes'_v\theta_{0,v})\neq 0$.

Let $w$ be a non-Archimedean place of $F$ such that   $|n|_w=1$ and $\Theta_r$ is unramified at $w$. Define a local functional $l_w:\Theta_{r,w}\to \bc$ by
\[
\theta_w \mapsto l(\theta_w \otimes (\otimes'_{v\neq w}\theta_{0,v} )).
\]
By our construction, $l_w$ is nonzero. Now the functional $l_w$ factors through the twisted Jacquet module of $\Theta_{r,w}$ for the character $\psi_{U_2(\sco)}(u)$ on the group $U_2(\sco)(F_w)$. This implies that $J_{U_2(\sco)(F_w),\psi_{U_2(\sco)}}(\Theta_{r,w})\neq 0$. This contradicts the local result.
\end{proof}

\subsection{Nonvanishing results}

\begin{Prop}\label{prop:nonvanishing global unipotent}
Let $\theta$ be in the space of $\Theta_r$. Let $\sco=(n^ab)$. Then the integral
\[
\int\limits_{U_2(\sco)(F)\bs U_2(\sco)(\ba)}\theta(ug)\psi_{U_2(\sco)}(u) \ du
\]
is nonzero for some choice of data.
\end{Prop}

\begin{proof}
The proof is analogous to the local case. Once we have the global root exchange lemma and global vanishing results, the nonvanishing results follow from the corresponding nonvanishing results on the semi-Whittaker coefficients. Notice that the global version of Lemma \ref{lem:vanishing unipotent orbit 1}, \ref{lem:vanishing unipotent orbit 2} and  \ref{lem:vanishing unipotent orbit 3} can be established by using the corresponding local results.  We omit the details. We remark that a similar argument can be found in \cite{JL2013} Section 5.
\end{proof}

\subsection{Unipotent orbits attached to Theta representations}

Finally we determine the unipotent orbit attached to $\Theta_r$.
\begin{Thm}\label{thm:global unipotent orbit}
The unipotent orbit attached to $\Theta_r$ is $(n^ab)$. In other words, $\sco(\Theta_r)=(n^ab)$.
\end{Thm}

\begin{proof}
By Definition \ref{def:unipotent orbit},  this follows from Propositions \ref{prop:vanishing global unipotent orbit} and \ref{prop:nonvanishing global unipotent}.
\end{proof}

\subsection{Whittaker-Speh-Shalika representations}\label{sec:WSS}

In the research announcement \cite{CFGK2016}, the famous doubling method is extended to the covering groups. A family of automorphic representations of $\tgl_r(\ba)$ are introduced as the induction data of Eisenstein series on covers of suitable split classical groups. We recall the definition here.

\begin{Def}\label{def1}
An irreducible genuine automorphic representation  $\pi$ of  $\tgl_{ab}({\ba})$ is a Whittaker-Speh-Shalika representation of type $(a,b)$ if:
\begin{enumerate}[\normalfont(1)]
\item\label{def:Whittaker-Speh-Shalika 1}
${\mathcal O}(\pi)=(a^b)$.
\item\label{def:Whittaker-Speh-Shalika 2}
For a finite place $v$, let $\pi_v$ denote the
irreducible constituent of $\pi$ at $v$.
Suppose that $\pi_v$ is an unramified representation. Then ${\mathcal O}(\pi_v)=(a^b)$.
(That is, the local analogue of part~(\ref{def:Whittaker-Speh-Shalika 1}) holds.) Moreover,
\begin{equation}
\dim{\Hom}_{U_2((a^b))(F_v)}(\pi_v,\psi_{U_2((a^b))})=1.
\end{equation}
\end{enumerate}
\end{Def}

Thus, when $r$ is a multiple of $n$, we can rephrase Theorem \ref{thm:local unipotent orbit 2}, \ref{thm:multiplicity one unipotent orbit} and \ref{thm:global unipotent orbit} as follows.

\begin{Thm}\label{thm:WSS}
When $r=mn$, the representation $\Theta_r$ is a Whittaker-Speh-Shalika representation of type $(n,m)$.
\end{Thm}




\bibliographystyle{siam}

\begin{thebibliography}{10}

\bibitem{BLS1999}
{\sc W.~D. Banks, J.~Levy, and M.~R. Sepanski}, {\em Block-compatible
  metaplectic cocycles}, J. Reine Angew. Math., 507 (1999), pp.~131--163.

\bibitem{BZ1977}
{\sc I.~N. Bernstein and A.~V. Zelevinsky}, {\em Induced representations of
  reductive {${\mathfrak p}$}-adic groups. {I}}, Ann. Sci. \'Ecole Norm. Sup.
  (4), 10 (1977), pp.~441--472.

\bibitem{BFG2003}
{\sc D.~Bump, S.~Friedberg, and D.~Ginzburg}, {\em Small representations for
  odd orthogonal groups}, Int. Math. Res. Not.,  (2003), pp.~1363--1393.

\bibitem{BG1992}
{\sc D.~Bump and D.~Ginzburg}, {\em Symmetric square {$L$}-functions on {${\rm
  GL}(r)$}}, Ann. of Math. (2), 136 (1992), pp.~137--205.

\bibitem{BGH1996}
{\sc D.~Bump, D.~Ginzburg, and J.~Hoffstein}, {\em The symmetric cube}, Invent.
  Math., 125 (1996), pp.~413--449.

\bibitem{CFGK2016}
{\sc Y.~Cai, S.~Friedberg, D.~Ginzburg, and E.~Kaplan}, {\em Doubling
  constructions for covering groups and tensor product ${L}$-functions},
  arXiv:1601.08240,  (2016).

\bibitem{CO2013}
{\sc G.~Chinta and O.~Offen}, {\em A metaplectic {C}asselman-{S}halika formula
  for {${\rm GL}_r$}}, Amer. J. Math., 135 (2013), pp.~403--441.

\bibitem{CM1993}
{\sc D.~H. Collingwood and W.~M. McGovern}, {\em Nilpotent orbits in semisimple
  {L}ie algebras}, Van Nostrand Reinhold Mathematics Series, Van Nostrand
  Reinhold Co., New York, 1993.

\bibitem{FKS1990}
{\sc Y.~Flicker, D.~Kazhdan, and G.~Savin}, {\em Explicit realization of a
  metaplectic representation}, J. Analyse Math., 55 (1990), pp.~17--39.

\bibitem{FG}
{\sc S.~Friedberg and D.~Ginzburg}, {\em Descent and theta functions for
  metaplectic groups}, Journal of the European Mathematical Society,  (to
  appear).

\bibitem{Gan}
{\sc W.~T. Gan}, {\em The metaplecitc tensor product as an instance of
  {L}anglands functoriality}, preprint,  (2016).

\bibitem{GJ1978}
{\sc S.~Gelbart and H.~Jacquet}, {\em A relation between automorphic
  representations of {${\rm GL}(2)$}\ and {${\rm GL}(3)$}}, Ann. Sci. \'Ecole
  Norm. Sup. (4), 11 (1978), pp.~471--542.

\bibitem{GK1982}
{\sc S.~S. Gelbart and A.~W. Knapp}, {\em {$L$}-indistinguishability and {$R$}\
  groups for the special linear group}, Adv. in Math., 43 (1982), pp.~101--121.

\bibitem{Ginzburg2006}
{\sc D.~Ginzburg}, {\em Certain conjectures relating unipotent orbits to
  automorphic representations}, Israel J. Math., 151 (2006), pp.~323--355.

\bibitem{Ginzburg2014}
\leavevmode\vrule height 2pt depth -1.6pt width 23pt, {\em Towards a
  classification of global integral constructions and functorial liftings using
  the small representations method}, Adv. Math., 254 (2014), pp.~157--186.

\bibitem{Ginzburg2015}
\leavevmode\vrule height 2pt depth -1.6pt width 23pt, {\em On certain global
  constructions of automorphic forms related to small representations of
  {$\mathrm{F}_4$}}, arXiv:1503.06409,  (2015).

\bibitem{Ginzburg2016}
\leavevmode\vrule height 2pt depth -1.6pt width 23pt, {\em Generating functions
  on covering groups}, arXiv:1603.05784,  (2016).

\bibitem{GRS1999}
{\sc D.~Ginzburg, S.~Rallis, and D.~Soudry}, {\em On a correspondence between
  cuspidal representations of {${\rm GL}_{2n}$} and {$\widetilde{\rm
  Sp}_{2n}$}}, J. Amer. Math. Soc., 12 (1999), pp.~849--907.

\bibitem{GRS2011}
\leavevmode\vrule height 2pt depth -1.6pt width 23pt, {\em The descent map from
  automorphic representations of {${\rm GL}(n)$} to classical groups}, World
  Scientific Publishing Co. Pte. Ltd., Hackensack, NJ, 2011.

\bibitem{JL2013}
{\sc D.~Jiang and B.~Liu}, {\em On {F}ourier coefficients of automorphic forms
  of {${\rm GL}(n)$}}, Int. Math. Res. Not. IMRN,  (2013), pp.~4029--4071.

\bibitem{Kable2001}
{\sc A.~C. Kable}, {\em The tensor product of exceptional representations on
  the general linear group}, Ann. Sci. \'Ecole Norm. Sup. (4), 34 (2001),
  pp.~741--769.

\bibitem{Kaplan}
{\sc E.~Kaplan}, {\em The characterization of theta-distinguished
  representations of $\mathrm{GL}_n$}, Israel J. Math,  (to appear).

\bibitem{KP1984}
{\sc D.~A. Kazhdan and S.~J. Patterson}, {\em Metaplectic forms}, Inst. Hautes
  \'Etudes Sci. Publ. Math.,  (1984), pp.~35--142.

\bibitem{McNamara2012}
{\sc P.~J. McNamara}, {\em Principal series representations of metaplectic
  groups over local fields}, in Multiple {D}irichlet series, {L}-functions and
  automorphic forms, vol.~300 of Progr. Math., Birkh\"auser/Springer, New York,
  2012, pp.~299--327.

\bibitem{Mezo2004}
{\sc P.~Mezo}, {\em Metaplectic tensor products for irreducible
  representations}, Pacific J. Math., 215 (2004), pp.~85--96.

\bibitem{Moeglin1996}
{\sc C.~M{\oe}glin}, {\em Front d'onde des repr\'esentations des groupes
  classiques {$p$}-adiques}, Amer. J. Math., 118 (1996), pp.~1313--1346.

\bibitem{MW1987}
{\sc C.~M{\oe}glin and J.-L. Waldspurger}, {\em Mod\`eles de {W}hittaker
  d\'eg\'en\'er\'es pour des groupes {$p$}-adiques}, Math. Z., 196 (1987),
  pp.~427--452.

\bibitem{MW1995}
\leavevmode\vrule height 2pt depth -1.6pt width 23pt, {\em Spectral
  decomposition and {E}isenstein series}, vol.~113 of Cambridge Tracts in
  Mathematics, Cambridge University Press, Cambridge, 1995.
\newblock Une paraphrase de l'\'Ecriture [A paraphrase of Scripture].

\bibitem{OS2008}
{\sc O.~Offen and E.~Sayag}, {\em Global mixed periods and local {K}lyachko
  models for the general linear group}, Int. Math. Res. Not. IMRN,  (2008),
  pp.~Art. ID rnm 136, 25.

\bibitem{PPS1989}
{\sc S.~J. Patterson and I.~I. Piatetski-Shapiro}, {\em The symmetric-square
  {$L$}-function attached to a cuspidal automorphic representation of {${\rm
  GL}_3$}}, Math. Ann., 283 (1989), pp.~551--572.

\bibitem{Shimura1975}
{\sc G.~Shimura}, {\em On the holomorphy of certain {D}irichlet series}, Proc.
  London Math. Soc. (3), 31 (1975), pp.~79--98.

\bibitem{Takeda2014}
{\sc S.~Takeda}, {\em The twisted symmetric square {$L$}-function of
  {${GL}(r)$}}, Duke Math. J., 163 (2014), pp.~175--266.

\bibitem{Takeda2016}
\leavevmode\vrule height 2pt depth -1.6pt width 23pt, {\em Metaplectic tensor
  products for automorphic representation of {$\widetilde{GL}(r)$}}, Canad. J.
  Math., 68 (2016), pp.~179--240.

\bibitem{Takeda2017}
\leavevmode\vrule height 2pt depth -1.6pt width 23pt, {\em Remarks on
  metaplectic tensor products for covers of {GL}(r)}, Pacific J. Math., 290
  (2017), pp.~199--230.

\end{thebibliography}




\end{document}